\documentclass[reqno]{amsart}
\usepackage[utf8]{inputenc}
\usepackage{amssymb}
\usepackage{amsmath}
\usepackage[margin=3cm]{geometry}
\usepackage{xcolor}
\usepackage{enumerate}   
\usepackage{float}
\usepackage{mathtools}
\usepackage{xcolor}
\usepackage{enumerate}   
\usepackage{tikz}
\usetikzlibrary{decorations.pathreplacing,
                calligraphy}
\tikzset{
B/.style = {decorate,
            decoration={calligraphic brace, amplitude=4pt,
            raise=1pt, mirror},
            pen colour=black},
dot/.style = {circle, fill, inner sep=2pt, outer sep=0pt}
        }
\usepackage{float}
\usepackage{mathtools}

\usepackage{indentfirst,amsthm,amscd}
\usepackage{amsthm,amsmath, amssymb, amsfonts}
\usepackage{hyperref}
\usepackage[capitalize,nameinlink]{cleveref}
\hypersetup{
    colorlinks,
    citecolor=red,
    filecolor=black,
    linkcolor=cyan,
    urlcolor=black
}

\newtheorem{theorem}{Theorem}[section]
\newtheorem*{kitai}{Kitai's Criterion}
\newtheorem{lemma}[theorem]{Lemma}
\newtheorem{corollary}[theorem]{Corollary}
\newtheorem{proposition}[theorem]{Proposition}
\theoremstyle{definition}

\newtheorem{example}[theorem]{Example}	
\newtheorem{remark}[theorem]{Remark}
\newtheorem*{acknowledgement}{Acknowledgement}

\usepackage{mathtools}

\title{Kitai's Criterion for Composition Operators}
\author{Daniel Gomes}
\address{Daniel Gomes, Departamento de Matemática, Instituto de Matemática, Estatística e Computação Científica, Universidade Estadual de Campinas, Rua Sérgio Buarque de Holanda, 651, 13083-970, Campinas-SP, Brazil}
\email{\href{mailto:danielgomes@ime.unicamp.br}{danielgomes@ime.unicamp.br}}

\author{Karl-G. Grosse-Erdmann}
\address{Karl-G. Grosse-Erdmann, Département de Mathématique, Université de Mons, 20 Place
du Parc, 7000 Mons, Belgium}
\email{\href{mailto:kg.grosse-erdmann@umons.ac.be}{kg.grosse-erdmann@umons.ac.be}}

\date{}
\begin{document}

\begin{abstract}
We present a general and natural framework to study the dynamics of composition operators on spaces of measurable functions, in which we then reconsider the characterizations for hypercyclic and mixing composition operators obtained by Bayart, Darji and Pires \cite{bayart2018topological}. We show that the notions of hypercyclicity and weak mixing coincide in this context and, if the system is dissipative, the recurrent composition operators agree with the hypercyclic ones. We also give a characterization for invertible composition operators satisfying Kitai's Criterion, and we construct an example of a mixing composition operator not satisfying Kitai's Criterion. For invertible dissipative systems with bounded distortion we show that composition operators satisfying Kitai's Criterion coincide with the mixing operators.
\end{abstract}
\maketitle

\section{Introduction}

The fact that there exist linear operators with a dense orbit is known since 1929, when Birkhoff \cite{birkhoff1929demonstration} proved that there are translation operators on the space of entire functions that are \textit{hypercyclic} (that is, they admit a dense orbit). Since then, many other kinds of chaotic behavior have been studied regarding the dynamics of linear operators, as one can see, for example, in the books \cite{BayartMatheron2009} and \cite{grosse2011linear}. Let us recall some of these notions that we shall study in this paper. We say that a continuous linear operator $T:\mathcal{X}\to \mathcal{X}$ acting on a Banach space $\mathcal{X}$ is:
\begin{itemize}
    \item \textit{recurrent} (see \cite{costakis2014}) if for every non-empty open set $U\subseteq \mathcal{X}$, there is some $k\geq 0$ (and hence infinitely many $k$) such that $T^k(U)\cap U \neq \varnothing$;
    \item \textit{topologically transitive} if for every pair $U,V\subseteq \mathcal{X}$ of non-empty open sets, there exists $k\geq 0$ such that $T^k(U)\cap V\neq \varnothing$;
    \item \textit{weakly mixing} if $T\times T$ is topologically transitive;
    \item \textit{topologically mixing} if for every pair $U,V\subseteq \mathcal{X}$ of non-empty open sets, there exists $k_0\geq 1$ such that $T^k(U)\cap V\neq \varnothing$ for every $k\geq k_0$.
\end{itemize}
 It is clear that 
\[
\text{topological mixing}\implies \text{weak mixing} \implies \text{topological transitivity} \implies \text{recurrence}. 
\]
By the Birkhoff transitivity theorem (see \cite[Theorem~1.16]{grosse2011linear}), topological transitivity is equivalent to hypercyclicity.

Many classes of linear operators have been studied lately. Motivated by the study of weighted shifts (a very important class of operators in linear dynamics), Bayart, Darji and Pires \cite{bayart2018topological} have studied the dynamics of composition operators on Banach spaces $\mathcal{X}$ of (equivalence classes of) measurable functions that generalize the spaces $L^p$.

We assume that $(X,\mathcal{B},\mu)$ is a $\sigma$-finite measure space and that $(\mathcal{X},\|\cdot\|)\subseteq L^0_\mu(X)$ is a Banach space of measurable functions. Given a measurable transformation $f:X\to X$, we consider the composition operator $T_f: L^0_\mu(X)\to L^0_\mu(X)$ given by $\varphi\mapsto \varphi\circ f$, where we assume that $f$ is a non-singular transformation, i.e. for every $B\in \mathcal{B}$, $\mu(B)=0$ implies $\mu(f^{-1}(B))=0$. This guarantees that the operator $T_f$ is well defined (see \cite{singh1993composition}). We then call $(X,\mathcal{B},\mu,f)$ a \textit{measurable system}. If $f$ is bijective and bimeasurable, then we assume that also $f^{-1}$ is non-singular so that $T_{f^{-1}}=T_f^{-1}$ is well defined. We then call $(X,\mathcal{B},\mu,f)$ an \textit{invertible measurable system}.
By the measurability of $f$, one always has that $f^{-1}(\mathcal{B})\subseteq \mathcal{B}$. We say that $f^{-1}(\mathcal{B})$ \textit{is essentially all of} $\mathcal{B}$, and denote it by $f^{-1}(\mathcal{B})=_{\text{ess}}\mathcal{B}$, if given $A\in \mathcal{B}$, there exists $B\in \mathcal{B}$ such that $\mu(A\Delta f^{-1}(B))=0$ (see \cite{whitley1978}). 

We will always assume that $T_f$ maps $\mathcal{X}$ into $\mathcal{X}$, so that we can study the dynamics of the composition operator $T_f$ on $\mathcal{X}$. For an invertible measurable system we also assume that $T_f^{-1}$ maps $\mathcal{X}$ into $\mathcal{X}$.

Bayart et al. \cite{bayart2018topological} have introduced four hypotheses (H1) to (H4) on $\mathcal{X}$ to study dynamical properties of the composition operator $T_f$. We propose here to consider instead the following a priori weaker, possibly equivalent, set of conditions, which we regard as a natural framework (for details we refer to \cref{sec: composition operators}):
\begin{enumerate}
    \item[(H1)] For any $A\in \mathcal{B}$ with finite measure, $\chi_A\in\mathcal{X};$
    \item[(H2)] The set of simple functions that vanish outside a set of finite measure is dense in $\mathcal{X};$
    \item[(C1)] The inclusion map $I:\mathcal{X}\xhookrightarrow{} L^0_\mu(X)$ is continuous;
    \item[(C2)] For every set $A\in \mathcal{B}$ of finite measure, the map
    \[
        J_A:\mathcal{B}(A)\to \mathcal{X}, \ C \mapsto \chi_C
    \]
    is continuous.
\end{enumerate}
Note that conditions (H2) and (C2) are understood to imply (H1). The conditions (H1) and (H2) are those from \cite{bayart2018topological}, while the continuity conditions (C1) and (C2) replace the previous conditions (H3) and (H4), respectively, as we will discuss below. 

In the following, the outer measure of a set $B\subseteq X$ is given by \[\mu^*(B)=\inf \{\mu(C): B\subseteq C, C\in \mathcal{B}\}.\]

In \cite[Theorem~1.1]{bayart2018topological}, it is shown that if $\mathcal{X}$ satisfies conditions (H1) to (H4), then $T_f:\mathcal{X}\to \mathcal{X}$ is hypercyclic if, and only if, $f^{-1}(\mathcal{B})=_{\text{ess}}\mathcal{B}$ and $f$ satisfies the \textit{Hypercyclic Runaway Condition}:
\begin{enumerate}
    \item[(HRC)] \textit{For every $A \in \mathcal{B}$ with finite measure there exist a sequence of measurable sets $B_k\subseteq A$, $k\geq 1$, and an increasing sequence $(n_k)_{k\geq 1}$ of positive integers such that
    \[
		\mu(A\setminus B_k)\to 0,\hspace*{0.5cm} \mu(f^{-n_k}(B_k))\to 0 \hspace*{0.5cm} \text{and} \hspace*{0.5cm} \mu^*(f^{n_k}(B_k))\to 0.
		\] }
\end{enumerate}
It is also shown in \cite[Theorem~1.2]{bayart2018topological} that $T_f$ is mixing if, and only if, $f^{-1}(\mathcal{B})=_{\text{ess}}\mathcal{B}$ and $f$ satisfies the \textit{Mixing Runaway Condition}:
\begin{enumerate}
    \item[(MRC)] \textit{For every $A \in \mathcal{B}$ with finite measure there exists a sequence of measurable sets $B_n\subseteq A$, $n\geq 1$, such that
    \[
\mu(A\setminus B_n)\to 0,\hspace*{0.5cm} \mu(f^{-n}(B_n))\to 0 \hspace*{0.5cm} \text{and} \hspace*{0.5cm} \mu^*(f^{n}(B_n))\to 0.
				\] }
\end{enumerate}

\begin{remark}\label{rem: BDP}
(a) We have corrected here a minor oversight in \cite{bayart2018topological}, where the condition $f^{-1}(\mathcal{B})=\mathcal{B}$ appears (see \cref{rmk: f-1(B)=B in bayart}).

(b) In \cite{bayart2018topological}, (HRC) was expressed in the following equivalent way: \textit{For every $A \in \mathcal{B}$ with finite measure and for every $\varepsilon>0$, there exist $B \in \mathcal{B}$ and $k\geq 1$ such that $B\subseteq A,$
    \begin{equation}\label{eq: HRC1}
		\mu(A\setminus B)<\varepsilon,\hspace*{0.5cm} \mu(f^{-k}(B))<\varepsilon \hspace*{0.5cm} \text{and} \hspace*{0.5cm} \mu^*(f^{k}(B))<\varepsilon.
		\end{equation}}
  Note that, by the proof in \cite{bayart2018topological}, $k$ can be chosen arbitrarily large.
		
(c) There is another equivalent but formally stronger condition: \textit{There exists an increasing sequence $(n_k)_{k\geq 1}$ of positive integers
such that, for every $A \in \mathcal{B}$ with finite measure, there exists a sequence of measurable sets $B_k\subseteq A$, $k\geq 1$, such that
    \begin{equation}\label{eq: HRC2}
		\mu(A\setminus B_k)\to 0,\hspace*{0.5cm} \mu(f^{-n_k}(B_k))\to 0 \hspace*{0.5cm} \text{and} \hspace*{0.5cm} \mu^*(f^{n_k}(B_k))\to 0.
		\end{equation}}
Indeed, since $X$ is $\sigma$-finite and therefore the union of an increasing sequence $(F_k)_{k\geq 1}$ of measurable sets with finite measure, by \eqref{eq: HRC1} there are measurable sets $E_k\subseteq F_k$, $k\geq 1$, and a sequence $(n_k)_k$ such that $\mu(F_k\setminus E_k)<\frac{1}{k}$, $\mu(f^{-{n_k}}(E_k))<\frac{1}{k}$ and $\mu^*(f^{n_k}(E_k))<\frac{1}{k}$, $k\geq 1$. Then, for any measurable set $A$ with finite measure, $\mu(A\setminus F_k)\to 0$ and hence the sets $B_k:=A\cap E_k$, $k\geq 1$, satisfy \eqref{eq: HRC2}.

(d) Condition (MRC) is also different from but equivalent to the one appearing in \cite{bayart2018topological}.
\end{remark}

A very useful tool to show that an operator is mixing is to show that it satisfies Kitai's Criterion, that is, it satisfies the hypotheses of the following theorem (see \cite{kitai1982}, \cite{grivaux2005hypercyclic}, \cite{grosse2011linear}):
\begin{kitai} 
    Let $T$ be an operator on a separable Banach space $\mathcal{X}$. If there are dense subsets $\mathcal{X}_0,\mathcal{Y}_0 \subseteq \mathcal{X}$ and a map $S:\mathcal{Y}_0 \to \mathcal{Y}_0$ such that, for any $x\in \mathcal{X}_0$ and $y\in \mathcal{Y}_0$
    \begin{enumerate}[\rm (i)]
        \item $T^nx \to 0,$
        \item $S^ny \to 0,$
        \item $TSy=y,$
    \end{enumerate}
    then $T$ is mixing.
\end{kitai}

The question regarding whether every mixing operator satisfies Kitai's Criterion, raised by Shapiro in \cite{shapiro2001notes}, was solved by Grivaux in \cite{grivaux2005hypercyclic}, where it is shown that every infinite dimension separable Banach space admits a mixing operator that does not satisfy Kitai's Criterion. By characterizing the invertible composition operators satisfying Kitai's Criterion (see \cref{principal}) and using the characterization for mixing composition operators given by  \cite[Theorem~1.2]{bayart2018topological}, we are able to construct an example of a composition operator that is mixing but does not satisfy Kitai's Criterion.

The paper is organised as follows. In \cref{sec: composition operators}, we will show in \cref{th: characterization hc,th: characterization mixing} that if one replaces (H3) and (H4) by (C1) and (C2), the characterizations for hypercyclic and mixing composition operators given in \cite{bayart2018topological} still hold. Even more, we will show that hypercyclic composition operators in this setting coincide with the weakly mixing ones. \cref{sec: kitais criterion} is devoted to our main result: we show that there exist dissipative composition operators, acting on $L^p(X), 1\leq p<\infty, $ that are mixing and do not satisfy Kitai's Criterion. \cref{sec: dissipative systems} is dedicated to the study of the dynamics of composition operators in the dissipative setting, where we introduce a weaker set of ``local'' conditions (LH1), (LH2), (LC1) and (LC2). In particular, we show that if the system is dissipative, recurrence implies hypercyclicity. In \cref{sec: backward shifts} we apply the main results of \cref{sec: dissipative systems} to backward shift operators on sequence spaces, thereby obtaining an improvement of the currently best result on hypercyclicity and mixing in the literature. Finally, \cref{sec: bounded distortion} is dedicated to the study of invertible dissipative systems with bounded distortion. In particular, we show that the notions of being mixing and satisfying Kitai's Criterion are equivalent in this context.

\begin{remark}\label{rmk: F-spaces}
All the results in this paper hold as well if $(\mathcal{X},\|\cdot\|)\subseteq L^0_\mu(X)$ is an F-space of measurable functions, where $\|\cdot\|$ denotes an F-norm (see \cite{kalton1984Fspaces}). In particular, they hold for all Fréchet spaces, with $\|\cdot\|$ an F-norm defining its topology (see \cite[p. 35]{grosse2011linear}). One need only note that any F-norm satisfies that $\|a\varphi\|\leq (|a|+1)\|\varphi\|$ for any scalar $a$.
\end{remark}

\section{Composition Operators}\label{sec: composition operators}

Let us begin this section explaining what we mean by the continuity of the inclusion in (C1), by defining a natural metric on $L^0_\mu(X)$. Let $(\varphi_n)_{n\geq 1}$ be a sequence of measurable functions. We say that $\varphi_n$ converges in measure to a measurable function $\varphi$, and denote it by $\varphi_n\xrightarrow[]{\mu}\varphi$, if for every $\xi>0$, 
\[
\mu(\{ x\in X: |\varphi_n(x)-\varphi(x)|\geq \xi\})\to 0.
\]
We will consider the extended metric $d$ on $L^0_\mu(X)$, compatible with convergence in measure, that was introduced by Fréchet (see \cite[Section 16]{frechet1921}, also \cite[p. 426]{Bogachev2007}) and is given by 
\[
d(\varphi,\psi)=\inf_{\xi>0}\Big(\mu(\{x\in X: |\varphi(x)-\psi(x)|\geq \xi\}) +\xi\Big). 
\]

The following lemma illustrates the relationship between the measure $\mu$ on $X$ and the distance $d$ on $L^0_\mu(X)$.

\begin{lemma}\label{lemma}
    Let $0<\varepsilon<1$ and $\varphi,\psi\in L^0_\mu(X)$. If $d(\varphi,\psi)<\varepsilon$, then
\[
\mu(\{x\in X: |\varphi(x)-\psi(x)|\geq 1\})<\varepsilon.
\]
\end{lemma}

\begin{proof}
   By definition of $d$ there exists $\xi>0$ such that
    \[
		\mu(\{x\in X: |\varphi(x)-\psi(x)|\geq \xi\}) +\xi< \varepsilon.
		\]
    In particular, $\xi<\varepsilon<1$ and
    \[
	\mu(\{x\in X: |\varphi(x)-\psi(x)|\geq 1\})\leq	\mu(\{x\in X: |\varphi(x)-\psi(x)|\geq \xi\})<\varepsilon.
		\]
\end{proof}

\begin{remark}\label{remark inclusion}
    Using \cref{lemma}, it is easy to see that (C1) is actually equivalent to (H3) in \cite{bayart2018topological}.
\end{remark}

Recall that we demand of $T_f$ that it maps the space $\mathcal{X}$ into itself.

\begin{proposition}
    Suppose that $\mathcal{X}$ satisfies the condition \emph{(C1)}. Then $T_f:\mathcal{X}\to\mathcal{X}$ is continuous.
\end{proposition}

\begin{proof}
Let us show that $T_f$ is continuous by using the Closed Graph Theorem. For that, let $(\varphi_n)_{n\geq 1}$ be a sequence of functions in $\mathcal{X}$ converging to $\varphi \in \mathcal{X}$ and such that $T_f \varphi_n$ converges to $\psi\in \mathcal{X}$. By (C1), we have that $\varphi_n\xrightarrow{\mu}\varphi$, so that there exists a subsequence $(n_k)_{k\geq 1}$ such that $\varphi_{n_k}\to \varphi$ almost everywhere (see \cite{royden1988realanalysis}), that is, there exists $A\in \mathcal{B}$ such that $\mu(X\setminus A)=0$ and \[\varphi_{n_k}(x)\to \varphi(x) \hspace*{0.5cm} \forall x\in A.\]

    Analogously, we have that there exists a subsequence $(n_{k_l})_{l\geq 1}$ and $B\in \mathcal{B}$ such that $\mu(X\setminus B)=0$ and \[\varphi_{n_{k_l}}\circ f(x)\to \psi(x) \hspace*{0.5cm} \forall x\in B.\]

    Since $f$ is non-singular, we have that $\mu(X\setminus f^{-1}(A))=\mu(f^{-1}(X\setminus A))=0$. Now let \[C=f^{-1}(A)\cap B.\]
    We have that $\mu(X\setminus C)=0$,
    \[\varphi_{n_{k_l}}\circ f(x)\to \varphi \circ f(x) \hspace*{0.5cm} \forall x\in C\]
    and 
    \[\varphi_{n_{k_l}}\circ f(x)\to \psi(x) \hspace*{0.5cm} \forall x\in C,\]
    so that $\psi(x)=\varphi \circ f(x)$ for all $x\in C$, showing that $T_f$ is continuous.
\end{proof}

Let us next explain what we mean by continuity in condition (C2). For every measurable set $A$ with finite measure, let 
\[
\mathcal{B}(A):=\{B\cap A:B\in \mathcal{B}\}.
\]
We then consider the pseudo-metric $d_A$ on $\mathcal{B}(A)$ given by 
\[
d_A(C,D)=\mu(C\Delta D), \hspace*{0.5cm} C,D\in \mathcal{B}(A).
\]

Condition (C2) can be expressed in a more useful equivalent way.

\begin{lemma}\label{lem-cont}
Condition \emph{(C2)} holds if, and only if, for any $\varepsilon>0$ there is some $\delta>0$ such that, for any sets $A,B\in\mathcal{B}$ with finite measure, 
\[
\mu(A\Delta B)<\delta \Longrightarrow \|\chi_A-\chi_B\|<\varepsilon.
\]
\end{lemma}

\begin{proof} It suffices to show that (C2) implies the stated stronger condition. Moreover, by definition of $A\Delta B$ and since $\chi_A-\chi_B = \chi_{A\setminus B} - \chi_{B\setminus A}$, it suffices to consider the case of $B=\varnothing$. Now, if the condition did not hold, one could find $\varepsilon>0$ and sets $A_n\in\mathcal{B}$ with $\mu(A_n)<n^{-2}$ such that $\|\chi_{A_n}\|\geq\varepsilon$, $n\geq 1$. But then (C2) would not hold for the set $A=\bigcup_nA_n$ of finite measure.
\end{proof}

As the proof shows it suffices, in (C2) and in the above condition, to have continuity at $B=\varnothing$.

\begin{remark}\label{remark h4}
Condition (H4) in \cite{bayart2018topological} can be expressed in the following way: for any $\varepsilon>0$ there is some $\delta>0$ such that, for any $\varphi\in\mathcal{X}$ and any set $A\in \mathcal{B}$, if $|\varphi|\leq \chi_A$ in the pointwise sense and $\mu(A)<\delta$ then $\|\varphi\|<\varepsilon$. Thus, (H4) clearly implies (C2). We do not know if the conditions are in fact equivalent.
\end{remark}

The following lemma can be found in Whitley \cite{whitley1978} if $\mathcal{X}=L^2(X,\mathcal{B},\mu)$, see also \cite[Corollary 2.2.8]{singh1993composition}; in the sequel we will only need part (a).

\begin{lemma}\label{Whitley}
\emph{(a)} Suppose that $\mathcal{X}$ satisfies conditions \emph{(H1)} and \emph{(C1)}. If $T_f$ has dense range then $f^{-1}(\mathcal{B})=_{\emph{\text{ess}}}\mathcal{B}$.

\emph{(b)} Suppose that $\mathcal{X}$ satisfies conditions \emph{(H2)} and \emph{(C2)}. If $f^{-1}(\mathcal{B})=_{\emph{\text{ess}}}\mathcal{B}$ then $T_f$ has dense range.
\end{lemma}

\begin{proof} For (a) one can follow the proof of \cite[Lemma~1]{whitley1978}. For (b), it suffices to show by (H2) that, for any set $A\in\mathcal{B}$ with finite measure, $\chi_A$ is in the closure of the range of $T_f$. By the hypothesis there is some set $B\in \mathcal{B}$ such that $\mu(A\Delta f^{-1}(B))=0$. Since $\mu$ is $\sigma$-finite, $B$ is the union of an increasing sequence $(B_k)_{k\geq 1}$ of measurable sets with finite measure. Then $\mu(A\Delta f^{-1}(B_k))=\mu((A\setminus f^{-1}(B_k))\cup(f^{-1}(B_k)\setminus A))=\mu(A\setminus f^{-1}(B_k))\to 0$, so that, by (C2) via \cref{lem-cont}, $T_f\chi_{B_k}=\chi_{B_k}\circ f=\chi_{f^{-1}(B_k)}\to \chi_A$ in $\mathcal{X}$, and $\chi_{B_k}\in \mathcal{X}$ by (H2).
\end{proof}

By \cref{remark inclusion}, \cref{Whitley}(a) and \cite[Theorem~1.1]{bayart2018topological}, it follows the next theorem. For the sake of completeness, we shall adapt the proof of \cite[Theorem~1.1]{bayart2018topological} to prove it explicitly using (C1) instead of (H3).

\begin{theorem}\label{th: hc implies scc}
    Suppose that $\mathcal{X}$ satisfies conditions \emph{(H1)} and \emph{(C1)}. If $T_f$ is hypercyclic, then $f^{-1}(\mathcal{B})=_{\emph{\text{ess}}}\mathcal{B}$ and $f$ satisfies \emph{(HRC)}.
\end{theorem}

\begin{proof}
    Since $T_f$ is hypercyclic, it has dense range, which implies that $f^{-1}(\mathcal{B})=_{\text{ess}}\mathcal{B}$ by \cref{Whitley}(a).

    Let $A\in \mathcal{B}$ be a set of finite measure and $0<\varepsilon<1$. By \cref{rem: BDP}(b) it suffices to find a measurable subset $B$ of $A$ and some $k\geq 1$ such that \eqref{eq: HRC1} holds. Now, by (H1) and (C1) there is some $\delta>0$ such that
    \begin{align*}
        \|\psi-2\chi_A\|<\delta &\implies d(\psi, 2\chi_A)<\frac{\varepsilon}{2},\\
        \|\psi-4\chi_A\|<\delta &\implies d(\psi, 4\chi_A)<\frac{\varepsilon}{2}.
    \end{align*}
    By hypercyclicity of $T_f$, there exist $k\geq 1$ and $\varphi\in \mathcal{X}$ such that
    \[\|\varphi-2\chi_A\|<\delta \hspace*{0.5cm} \text{and} \hspace*{0.5cm} \|\varphi\circ f^k-4\chi_A\|<\delta.\]

    By \cref{lemma}, we have that 
    \begin{equation}\label{eq: measure of sets}
        \mu(\{x\in X: |\varphi(x)-2\chi_A(x)|\geq 1\})<\frac{\varepsilon}{2} \hspace*{0.5cm} \text{and} \hspace*{0.5cm} \mu(\{x\in X: |\varphi\circ f^k(x)-4\chi_A(x)|\geq 1\})<\frac{\varepsilon}{2}.
    \end{equation}

    Let \[C=\{x\in X: |\varphi (x)-4|<1\}\hspace*{0.5cm} \text{and} \hspace*{0.5cm} 
    D=\{x\in X: |\varphi (x)-2|<1\},\]
    and define \[B=D\cap f^{-k}(C)\cap A.\]

    First, note that \[A\setminus D\subseteq \{x\in X: |\varphi(x)-2\chi_A(x)|\geq 1\}\] and \[A\setminus f^{-k}(C)\subseteq \{x\in X: |\varphi\circ f^k(x)-4\chi_A(x)|\geq 1\},\]
    so \labelcref{eq: measure of sets} implies that $\mu(A\setminus B)<\varepsilon$.

    Next, we have that
    \[ f^{-k}(D)\subseteq \{x\in X: |\varphi \circ f^k(x)-4 \chi_A(x)|\geq 1\}, \]
    and hence, by \labelcref{eq: measure of sets}, $\mu(f^{-k}(B))\leq \mu (f^{-k}(D))<\varepsilon$.

    Finally, \[C\subseteq \{x\in X: |\varphi(x)-2\chi_A(x)|\geq 1\}.\]
    Since $f^k(B)\subseteq C$, by \labelcref{eq: measure of sets} we have that $\mu^*(f^k(B))<\varepsilon$.
\end{proof}

In the next theorem, we will improve \cite[Theorem~1.1]{bayart2018topological} by showing that (HRC) implies not only hypercyclicity, but weak mixing under conditions (H2) and (C2). We first note that, by the triangle inequality for $(A,B)\mapsto\mu(A\Delta B)$ and since $f$ is non-singular, we have the following:

\begin{lemma}\label{lemma: pseudometric}
    If $f^{-1}(\mathcal{B})=_{\emph{\text{ess}}}\mathcal{B}$, then for every $A\in \mathcal{B}$ and $k\geq 1$, there exists $B\in \mathcal{B}$ such that \[\mu(f^{-k}(B)\Delta A)=0.\]
\end{lemma}

We will use a standard notation: given $U,V\subseteq \mathcal{X}$, let 
\[
N(U,V)=\{n\in \mathbb{N}_0: T_f^n(U)\cap V\neq \varnothing\}.
\]
Hence, $T_f$ is weakly mixing if, and only if, for every non-empty open sets $U,V,\Tilde{U},\Tilde{V}\subseteq \mathcal{X}$, we have that \[N(U,V)\cap N(\Tilde{U},\Tilde{V})\neq \varnothing.\]

\begin{theorem}\label{th: ssc implies hc}
    Suppose that $\mathcal{X}$ satisfies conditions \emph{(H2)} and \emph{(C2)}. Assume that $f^{-1}(\mathcal{B})=_{\emph{\text{ess}}}\mathcal{B}$ and that $f$ satisfies \emph{(HRC)}. Then $T_f$ is weakly mixing.
\end{theorem}

\begin{proof}
    Let $U,V,\Tilde{U},\Tilde{V}$ be non-empty open subsets of $\mathcal{X}$. By (H2), there exist $A\in \mathcal{B}$ with finite measure, simple functions $\psi_1=\sum_{j=1}^{M} a_j \chi_{A_j}$, $\psi_2=\sum_{j=1}^{M} b_j \chi_{B_j}$, $\Tilde{\psi}_1=\sum_{j=1}^{M} \Tilde{a}_j \chi_{\Tilde{A}_j}$, $\Tilde{\psi}_2=\sum_{j=1}^{M} \Tilde{b}_j \chi_{\Tilde{B}_j}$ and $\eta>0$ such that $A_j,B_j,\Tilde{A}_j, \Tilde{B}_j\subseteq A$,
    \[
		B(\psi_1,\eta)\subseteq U, \hspace*{0.5cm} B(\psi_2,\eta)\subseteq V, \hspace*{0.5cm} B(\Tilde{\psi}_1,\eta)\subseteq \Tilde{U} \hspace*{0.5cm} \text{and} \hspace*{0.5cm} B(\Tilde{\psi}_2,\eta)\subseteq \Tilde{V},
		\] 
where $B(\varphi,\eta)$ denotes the open ball of radius $\eta$ around $\varphi$. Let $L=1+\sum_j(|a_j|+|b_j|+|\Tilde{a}_j |+|\Tilde{b}_j|)$.

Condition (C2) implies, via \cref{lem-cont}, that there exists $\varepsilon>0$ such that, for any sets $E,F\in \mathcal{B}$ with finite measure,
    \begin{equation} \label{eq: continuity of J}
        \mu(E\Delta F)<2\varepsilon \implies \|\chi_E-\chi_F\|<\frac{\eta}{4L}.
    \end{equation}

Now, by (HRC), there are some $B\in \mathcal{B}$, $B\subseteq A$, and $k\geq 1$ such that 
\[
		\mu(A\setminus B)<\varepsilon,\hspace*{0.5cm} \mu(f^{-k}(B))<\varepsilon \hspace*{0.5cm} \text{and} \hspace*{0.5cm} \mu^*(f^{k}(B))<\varepsilon,
\]
see \cref{rem: BDP}(b). We will show that $k\in N(U,V)\cap N(\Tilde{U},\Tilde{V})$. For that, let $C\in \mathcal{B}$ be such that $f^k(B)\subseteq C$ and $\mu(C)<\varepsilon$. 
		
		Let $\gamma_1=\sum_j a_j \chi_{A_j\cap B}$. Since
\begin{equation}\label{eq:1}
\mu(A_j\Delta (A_j\cap B))=\mu(A_j\setminus B)\leq \mu (A\setminus B)<\varepsilon\leq 2\varepsilon,
\end{equation}
    we have by \labelcref{eq: continuity of J} that
\begin{equation}\label{eq:2}
\|\chi_{A_j}-\chi_{A_j\cap B}\|<\frac{\eta}{4L}\leq \frac{\eta}{2L}.
\end{equation}
    By the triangle inequality, 
\begin{equation}\label{eq:3}
\|\psi_1-\gamma_1\|<\frac{\eta}{2}.
\end{equation}
    If we define $\gamma_2=\sum_j b_j \chi_{B_j\cap B}$, by a similar argument we get $\|\psi_2-\gamma_2\|<\eta/2$.

    Now since $f^{-1}(\mathcal{B})=_{\text{ess}}\mathcal{B}$, by \cref{lemma: pseudometric}, for every $j$ there exists $D_j\in \mathcal{B}$ such that \[\mu(f^{-k}(D_j)\Delta (B_j\cap B))=0.\] 
Note that $f^{-k}(D_j)$ has finite measure. 
    Let $C_j=D_j\cap C$ and define $\varphi$ by:
    \begin{equation}\label{eq: definition of phi}
        \varphi=\sum_j a_j\chi_{(A_j\cap B)\setminus \bigcup_i C_i}+\sum_j b_j \chi_{C_j}. 
    \end{equation}

    First, let us show that $\varphi\in U$. Since we have that each $C_j\subseteq C$,
    \[\mu\Big((A_j\cap B)\Delta \Big((A_j\cap B)\setminus \bigcup_i C_i\Big)\Big)=\mu \Big(A_j\cap B \cap \bigcup_i C_i\Big)\leq \mu\Big(\bigcup_i C_i\Big)\leq \mu(C)<\varepsilon \]
    and \[\mu(\varnothing \Delta C_j)=\mu(C_j)\leq \mu(C)<\varepsilon\leq 2\varepsilon,\]
    thus, by \labelcref{eq: continuity of J}, \eqref{eq:1}, \eqref{eq:2} and the triangle inequality for $(A,B)\mapsto\mu(A\Delta B)$, we obtain 
    \[\|\chi_{A_j\cap B}-\chi_{(A_j\cap B)\setminus \bigcup_i C_i}\|<\frac{\eta}{2L} \hspace*{0.5cm} \text{and} \hspace*{0.5cm} \|\chi_{C_j}\|<\frac{\eta}{4L}\leq\frac{\eta}{2L}. \]
    
    Hence, using the triangle inequality,
    \[\|\gamma_1-\varphi\|<\frac{\eta}{2},\]
    showing with \eqref{eq:3} that $\varphi\in U$. 
    
    To see that $\varphi \circ f^k\in V$, note that, by \labelcref{eq: definition of phi},
    \[\varphi\circ f^k=\sum_j a_j\chi_{f^{-k}((A_j\cap B)\setminus \bigcup_i C_i)}+\sum_j b_j \chi_{f^{-k}(C_j)}. \]
    Since 
    \[\mu\Big(\varnothing \Delta \Big(f^{-k}\Big((A_j\cap B)\setminus \bigcup_i C_i\Big)\Big)\Big)\leq \mu(f^{-k}(B))<\varepsilon \]
    and \[\mu(f^{-k}(C_j) \Delta (B_j\cap B))\leq \mu(f^{-k}(D_j) \Delta (B_j\cap B))=0,\]
    where $f^{-k}(C_j)$ has finite mesure,
    by the same argument as done above we get that \[\|\gamma_2-\varphi \circ f^k\|<\frac{\eta}{2},\]
    which shows that $\varphi \circ f^k\in V$. We have therefore shown that $k\in N(U,V)$.

    Now it remains to show that $k\in N(\Tilde{U},\Tilde{V}).$ The construction will be analogous to the previous one. Define $\Tilde{\gamma}_1=\sum_j \Tilde{a}_j \chi_{\Tilde{A}_j \cap B}$ and $\Tilde{\gamma}_2=\sum_j \Tilde{b}_j \chi_{\Tilde{B}_j \cap B}$, so that $\|\Tilde{\psi}_1-\Tilde{\gamma}_1\|<\eta/2$ and $\|\Tilde{\psi}_2-\Tilde{\gamma}_2\|<\eta/2$. For every $j$, let $\Tilde{D}_j\in \mathcal{B}$ be such that \[\mu(f^{-k}(\Tilde{D}_j)\Delta (\Tilde{B}_j\cap B))=0.\] Let $\Tilde{C}_j=\Tilde{D}_j\cap C$ and define $\Tilde{\varphi}$ by
    \[\Tilde{\varphi}=\sum_j \Tilde{a}_j\chi_{(\Tilde{A}_j\cap B)\setminus \bigcup_i \Tilde{C}_i}+\sum_j \Tilde{b}_j \chi_{\Tilde{C}_j}.\]
    By the same argument as above, we have that $\Tilde{\varphi}\in \Tilde{U}$ and $\Tilde{\varphi}\circ f^k\in \Tilde{V}$, so that also $k\in N(\Tilde{U},\Tilde{V}).$
\end{proof}

\begin{example}\label{rmk: f-1(B)=B in bayart}
    In \cref{th: hc implies scc}, we  cannot use the hypercyclicity of a composition operator $T_f$ in order to conclude that $f^{-1}(\mathcal{B})=\mathcal{B}$, as stated in \cite[Theorem~1]{bayart2018topological}. Indeed, consider $X=\mathbb{Z}\cup \{a\}$ to be the set of integers $\mathbb{Z}$ when we add a point $a\notin \mathbb{Z}$ to it, and define a measure $\mu$ on $\mathcal{P}(X)$ by 
\[
\mu(\{x\})=\begin{cases}
        1/x^2 & x\in X\setminus \{0,a\},\\
        1 & x=0,\\
        0 & x=a.
    \end{cases}
\]
    Let $f(x)=x+1$ if $x\neq a$ and $f(a)=0$, so that $T_{f|_\mathbb{Z}}:\ell^2(\mathbb{Z},\mu)\to \ell^2(\mathbb{Z},\mu)$ is the backward shift, which is hypercyclic (see \cite{grosse2011linear}). Since $\mu(\{a\})=0$, also $T_f:\ell^2(X,\mu)\to \ell^2(X,\mu)$ is hypercyclic. Now suppose that there is a $C\in \mathcal{P}(X)$ such that $f^{-1}(C)=\{a\}$. Since $f(a)=0$, we have that $0\in C$. But now $\{a\}=f^{-1}(C)\supseteq f^{-1}(\{0\})=\{-1,a\}$, which is a contradiction.
\end{example}

By \cref{th: hc implies scc,th: ssc implies hc}, we have now an improved version of \cite[Theorem~1.1]{bayart2018topological}.

\begin{theorem}\label{th: characterization hc}
    Suppose that $\mathcal{X}$ satisfies conditions \emph{(H1)}, \emph{(H2)}, \emph{(C1)} and \emph{(C2)}. The following are equivalent:
    \begin{itemize}
        \item[(a)] $T_f$ is hypercyclic;
        \item[(b)] $T_f$ is weakly mixing;
        \item[(c)] $f^{-1}(\mathcal{B})=_{\emph{\text{ess}}}\mathcal{B}$ and $f$ satisfies \emph{(HRC)}.
    \end{itemize}
\end{theorem}

We have analogous results to \cref{th: hc implies scc,th: ssc implies hc,th: characterization hc} for mixing operators, replacing ``hypercyclic" by ``mixing" and ``(HRC)" by ``(MRC)". In particular, we have the following, which implies \cite[Theorem~1.2]{bayart2018topological}.

\begin{theorem}\label{th: characterization mixing}
    Suppose that $\mathcal{X}$ satisfies conditions \emph{(H1)}, \emph{(H2)}, \emph{(C1)} and \emph{(C2)}. The following are equivalent:
    \begin{itemize}
        \item[(a)] $T_f$ is mixing;
        \item[(b)] $f^{-1}(\mathcal{B})=_{\emph{\text{ess}}}\mathcal{B}$ and $f$ satisfies \emph{(MRC)}.
    \end{itemize}
\end{theorem}

\section{Kitai's Criterion}\label{sec: kitais criterion}
In this section we will give a characterization for composition operators satisfying Kitai's Criterion, when the measurable system is invertible. To this end, we have the following two theorems, analogous to \cref{th: ssc implies hc,th: hc implies scc}, but now first supposing that $f$ has a measurable, non-singular \textit{right} inverse and then a measurable, non-singular \textit{left} inverse.

\begin{theorem}\label{thm: KitaiNec}
    Suppose that $\mathcal{X}$ satisfies \emph{(H1)} and \emph{(C1)}. Suppose also that there exists a non-singular measurable function $g$ such that $f\circ g=id$. If $T_f$ satisfies Kitai's Criterion, then for every $A\in \mathcal{B}$ with finite measure and for every $\varepsilon>0$, there exists a measurable set $B\subseteq A$ such that
    \[
    \mu(A\setminus B)<\varepsilon, \hspace*{0.5cm} \mu(f^{-n}(B))\to 0 \hspace*{0.5cm} \text{and} \hspace*{0.5cm} \mu(g^{-n}(B))\to 0.
    \]
\end{theorem}

\begin{proof}
    Consider $\mathcal{X}_0, \mathcal{Y}_0$ to be dense sets and $S:\mathcal{Y}_0\to \mathcal{Y}_0$ as in the statement of Kitai's Criterion. Let $A$ be a measurable set with finite measure and $0<\varepsilon<1$. By (H1) and (C1) there is some $\delta>0$ such that
    \[
    \|\psi-2\chi_A\|<\delta \implies d(\psi, 2\chi_A)<\frac{\varepsilon}{2}.
    \]
    
    By hypothesis, there exists $\varphi\in \mathcal{X}_0$ and $\Tilde{\varphi}\in \mathcal{Y}_0$ such that 
    \[
    \|\varphi-2\chi_A\|<\delta \hspace*{0.5cm} \text{and} \hspace*{0.5cm} \|\Tilde{\varphi}-2\chi_A\|<\delta.
    \]
    Now let 
    \[
    	C=\{x\in X:|\varphi(x)-2|< 1\} \hspace*{0.5cm} \text{and} \hspace*{0.5cm} C'=\{x\in X:|\Tilde{\varphi}(x)-2|< 1\}.
    \]
    We will show that $B=C\cap C'\cap A$ satisfies the properties we want.

    First, we have that 
    \[
        A \setminus C\subseteq \{x\in X: |\varphi(x)-2\chi_A(x)|\geq 1\}.
    \]
    By \cref{lemma} we have that $\mu(A\setminus C)<\varepsilon/2$. Analogously, $\mu(A\setminus C')<\varepsilon/2$, so that $\mu(A\setminus B)<\varepsilon$.

    We now have that 
    \[
        f^{-n}(B)\subseteq f^{-n}(C)\subseteq \{x\in X: |\varphi(f^n(x))|\geq 1\}.
    \]
    Let $0<\xi<1$ be given. By (C1) there is some $\eta>0$ such that 
    \[
    \|\psi\|<\eta \implies d(\psi,0)<\xi.
    \]
    Let $N\geq 1$ be such that $\|T^n_f\varphi\|<\eta$ for all $n\geq N$. By \cref{lemma}, we have that 
    \[
    \mu(f^{-n}(B))\leq \mu(\{x\in X: |\varphi(f^n(x))|\geq 1\})<\xi \hspace*{0.5cm} \forall n\geq N,
    \]
    showing that $\mu(f^{-n}(B))\to 0$.  

    Now since $T_f S \Tilde{\varphi}=\Tilde{\varphi}$ for every $\Tilde{\varphi}\in \mathcal{Y}_0$ and $f\circ g=id$, we have
    \[
    S\Tilde{\varphi}=(S\Tilde{\varphi})\circ f \circ g=(T_f S\Tilde{\varphi})\circ g=\Tilde{\varphi}\circ g \hspace*{0.5cm} \forall \Tilde{\varphi}\in \mathcal{Y}_0,
    \]
    so that $S=T_g$ on $\mathcal{Y}_0$. Note that these equalities are well defined in $L^0_\mu(X)$ because $f$, $g$ and then also $f\circ g$ are non-singular. Applying the same arguments as before with $g$ instead of $f$, we can conclude that $\mu(g^{-n}(B))\to 0$.
\end{proof}

\begin{theorem}\label{thm: Kitai}
    Suppose that $\mathcal{X}$ satisfies \emph{(H2)} and \emph{(C2)}. Suppose also that there exists a non-singular measurable function $g$ such that $g\circ f=id$. If, for every $A\in \mathcal{B}$ with finite measure and for every $\varepsilon>0$, there exists a measurable set $B\subseteq A$ such that
    \[
    \mu(A\setminus B)<\varepsilon, \hspace*{0.5cm} \mu(f^{-n}(B))\to 0 \hspace*{0.5cm} \text{and} \hspace*{0.5cm} \mu(g^{-n}(B))\to 0,
    \]
    then $T_f$ satisfies Kitai's Criterion.
\end{theorem}

\begin{proof}
    Let $\mathcal{X}_0$ be the set of simple functions $\sum_{j=1}^{M} b_j \chi_{B_j}$, where $b_j$ are scalars and each $B_j$ is a measurable set that satisfies $\mu(f^{-n}(B_j))\to 0$ and $\mu(g^{-n}(B_j))\to 0.$ 
    We claim that $\mathcal{X}_0$ is dense in $\mathcal{X}$. 
    
    Indeed, let $\varepsilon>0$ and $\varphi=\sum_{j=1}^{M} a_j \chi_{A_j}$ be a simple function such that each $A_j$ has finite measure. Using (C2) via \cref{lem-cont}, let $\delta>0$ be such that 
    \[
    \mu(C)<\infty, \hspace*{0.3cm}\mu(A_j \Delta C)<\delta \implies \|\chi_{A_j}-\chi_C\|<\frac{\varepsilon}{1+\sum_{k=1}^{M}|a_k|}, \hspace*{0.5cm} 1\leq j\leq M.
    \]
    By hypothesis, for each $A_j$, there exists a measurable set $B_j\subseteq A_j$ such that 
    \[
    \mu(A_j\setminus B_j)<\delta, \hspace*{0.5cm} \mu(f^{-n}(B_j))\to 0 \hspace*{0.5cm} \text{and} \hspace*{0.5cm} \mu(g^{-n}(B_j))\to 0.
    \] 
    
    Consider the function $\Tilde{\varphi}=\sum_{j=1}^{M} a_j \chi_{B_j}\in \mathcal{X}_0$. We have that
    \[
        \|\varphi-\Tilde{\varphi}\|=\Big\|\sum_{j=1}^{M} a_j(\chi_{A_j}-\chi_{B_j})\Big\|\\
        \leq \sum_{j=1}^{M}|a_j| \|\chi_{A_j}-\chi_{B_j}\|<\varepsilon,
    \]
    which completes the proof of the claim in view of (H2).

    Let $\varphi\in \mathcal{X}_0$ and $\varepsilon>0$. We have that there are scalars $b_j$ and measurable sets $B_j$, $1\leq j\leq M$, such that $\mu(f^{-n}(B_j))\to 0 $ and $\mu(g^{-n}(B_j))\to 0$, and $\varphi = \sum_{j=1}^{M} b_j \chi_{B_j}$. Let $\delta>0$ be such that
    \[
\mu(C)<\delta \implies \|\chi_C\|<\frac{\varepsilon}{1+\sum_{j=1}^{M} |b_j|}.
    \]
    Let $N_0\geq 1$ be such that, for all $1\leq j\leq M$, 
    \[
    \mu(f^{-n}(B_j))<\delta \hspace*{0.5cm} \forall n\geq N_0.
    \]
    
    We have that, for every $n\geq N_0$,
    \begin{align*}
        \|T_f^{n}\varphi\|&=\Big\| \sum_{j=1}^{M} b_j \chi_{B_j} \circ f^n\Big\|=\Big\| \sum_{j=1}^{M} b_j \chi_{f^{-n}(B_j)}\Big\| \\
        &\leq \sum_{j=1}^{M} |b_j| \|\chi_{f^{-n}(B_j)}\|<\varepsilon,
    \end{align*}
    showing that $T^n_f\varphi\to 0$. 

    We take $\mathcal{Y}_0=\mathcal{X}_0$ and $S=T_g$. If $\Tilde{\varphi}=\sum_{j=1}^{M} b_j \chi_{B_j}\in \mathcal{Y}_0$ then
    \[
    S\Tilde{\varphi}=\sum_{j=1}^{M} b_j \chi_{B_j}\circ g = \sum_{j=1}^{M} b_j \chi_{g^{-1}(B_j)}.
    \]
    Since, by hypothesis, $f^{-n}(g^{-1}(B_j))=f^{-(n-1)}(f^{-1}(g^{-1}(B_j))) =  f^{-(n-1)}(B_j)$ and $g^{-n}(g^{-1}(B_j))= g^{-(n+1)}(B_j)$, it follows that $S:\mathcal{Y}_0\to \mathcal{Y}_0$.     
    Using now the hypothesis $\mu(g^{-n}(B))\to 0$, the same arguments as above show that $S^n\Tilde{\varphi}\to 0$ for every $\Tilde{\varphi}\in \mathcal{Y}_0$.

    Finally, in view of the non-singularity of $f$, $g$, and then $g\circ f$, we have
    \[
    T_f S \Tilde{\varphi}= (S\Tilde{\varphi})\circ f = \Tilde{\varphi} \circ g \circ f = \Tilde{\varphi}, \hspace*{0.5cm} \forall \Tilde{\varphi}\in \mathcal{Y}_0,
    \]
    showing that $T_f$ satisfies Kitai's Criterion.
\end{proof}
            
\begin{remark}
(a) The proof of \cref{thm: Kitai} shows that $T_f$ satisfies a strong form of Kitai's Criterion where 
$\mathcal{Y}_0=\mathcal{X}_0$.

 (b) The proofs of the previous two theorems are analogous to the proof of \cite[Theorem~3.1]{darji2021}, where invertible composition operators on $L^p(X)$, $p\geq 2$, satisfying the Frequent Hypercyclicity Criterion are characterized.
\end{remark}

If $(X,\mathcal{B},\mu,f)$ is an invertible measurable system, we say that $f$ satisfies the \textit{Kitai Runaway Condition} if it satisfies the following condition:
\begin{enumerate}
    \item[(KRC)] \textit{For every $A\in \mathcal{B}$ with finite measure and for every $\varepsilon>0$, there exists a measurable set $B\subseteq A$ such that
    \[
    \mu(A\setminus B)<\varepsilon, \hspace*{0.5cm} \mu(f^{-n}(B))\to 0 \hspace*{0.5cm} \text{and} \hspace*{0.5cm} \mu(f^n(B))\to 0.
    \] }
\end{enumerate}

Combining these two theorems, we have the following characterization for invertible composition operators satisfying Kitai's Criterion:

\begin{theorem}\label{thm: kitai characterization}
Let $(X,\mathcal{B},\mu,f)$ be an invertible measurable system. Suppose that $\mathcal{X}$ satisfies \emph{(H1)}, \emph{(H2)}, \emph{(C1)} and \emph{(C2)}. The following are equivalent:
\begin{itemize}
    \item[(a)] $T_f$ satisfies Kitai's Criterion;
    \item[(b)] $f$ satisfies \emph{(KRC)}.
\end{itemize}
\end{theorem}

Using this characterization and \cref{th: characterization mixing}, we are able to construct an example of a mixing composition operator that does not satisfy Kitai's Criterion:

\begin{theorem}\label{principal}
    There exists an invertible measurable system $(X,\mathcal{B},\mu,f)$ such that $T_f:L^p(X)\to L^p(X)$, $1\leq p <\infty$, is mixing and does not satisfy Kitai's Criterion.
\end{theorem}

\begin{proof}
    Consider $[0,1]$ with the usual topology and with the Borel $\sigma$-algebra, and let $\lambda$ be the Lebesgue measure on $[0,1]$. Let $I_n=\{n\}\times [0,1]$ for every $n\in \mathbb{Z}$. We consider $I_n$ with the product $\sigma$-algebra, and with the measure $\mu_n$ defined by the following: if $C\subseteq [0,1]$ is measurable, we define 
    \[
    \mu_n(\{n\}\times C)=\frac{1}{2^{|n|}}\lambda(C), \hspace*{0.5cm} n\in \mathbb{Z}.
    \] 
    For every $n\in \mathbb{N}$ let $\alpha_{n,j}^l$, $1\leq l \leq 4n$, $1\leq j\leq 2^n$, be real numbers satisfying 
    \begin{equation}\label{eq: alpha}
         n<\alpha_{n,1}^1<\alpha_{n,1}^2<\ldots<\alpha_{n,1}^{4n}<\alpha_{n,2}^1<\alpha_{n,2}^2<\ldots<\alpha_{n,2^n}^{4n}<n+1, 
    \end{equation}
    and let $I_{n,j}^l=\{\alpha_{n,j}^l\}\times [0,1].$ We again consider $I_{n,j}^l$ with the product $\sigma$-algebra, but now we define the measure $\mu_{n,j}^l$ on $I_{n,j}^l$ in the following way: if $C\subseteq [0,1]$ is measurable, then 
    \[
    \mu_{n,j}^l(\{\alpha_{n,j}^l\}\times C)=\begin{dcases}
        \frac{2^l}{2^n}\lambda \Big( C\cap \Big[\frac{j-1}{2^n},\frac{j}{2^n}\Big]\Big)+\frac{1}{2^n}\lambda \Big( C\setminus \Big[\frac{j-1}{2^n},\frac{j}{2^n}\Big]\Big), & 1\leq l\leq 2n,\\
        \frac{2^{4n-l}}{2^n}\lambda \Big( C\cap \Big[\frac{j-1}{2^n},\frac{j}{2^n}\Big]\Big)+\frac{1}{2^n}\lambda \Big( C\setminus \Big[\frac{j-1}{2^n},\frac{j}{2^n}\Big]\Big), & 2n+1\leq l \leq 4n,
    \end{dcases}
    \]
    where $1\leq j\leq 2^n$.

    Let \[
    X=\bigcup_{n\in \mathbb{Z}}I_n \cup \bigcup_{\substack{n\in \mathbb{N}\\ 1\leq l\leq 4n \\ 1\leq j \leq 2^n}} I_{n,j}^l.
    \]
    Consider on $X$ the $\sigma$-algebra 
    \[
    \mathcal{B}=\Big\{\Big(\bigcup_{n\in \mathbb{Z}}\{n\}\times C_n\Big) \cup \Big(\bigcup_{\substack{n\in \mathbb{N}\\ 1\leq l\leq 4n \\ 1\leq j \leq 2^n}} \{\alpha_{n,j}^l\} \times C_{n,j}^l\Big):C_n, C_{n,j}^l\subseteq [0,1] \text{ measurable}\Big\}.  
    \]
    The measure $\mu$ on $X$ is defined by
    \[
    \mu(B)=\sum_{n\in \mathbb{Z}} \mu_n(B\cap I_n)+ \sum_{\substack{n\in \mathbb{N}\\ 1\leq l\leq 4n \\ 1\leq j \leq 2^n}} \mu_{n,j}^l(B\cap I_{n,j}^l).
    \]

    Define the function $f:X\to X$ by $f(n,x)=(n+1,x)$ if $n\leq 0$, $f(n,x)=(\alpha_{n,1}^1,x)$ if $n\geq 1$ and 
    \[
    f(\alpha_{n,j}^l,x)=\begin{dcases}
        (\alpha_{n,j}^{l+1},x), &\text{ if } 1\leq l\leq 4n-1 \text{ and } 1\leq j\leq 2^n,\\
        (\alpha_{n,j+1}^{1},x),  &\text{ if } l=4n \text{ and } 1\leq j\leq 2^n-1,\\
        (n+1,x), & \text{ if } l=4n \text{ and } j=2^n,
    \end{dcases}
    \]
    compare with \eqref{eq: alpha}.
    
    We have that each $I_n$, $n\in \mathbb{Z},$ has measure $1/2^{|n|}$. For the sets $I_{n,j}^l$, the subset given by $\{\alpha_{n,j}^l\}\times [(j-1)/2^n,j/2^n]$ has measure $2^l/2^{2n}$ if $1\leq l \leq 2n$, and it has measure $2^{4n-l}/2^{2n}$ if $2n+1 \leq l \leq 4n$. The complement of this subset in $I_{n,j}^l$ has measure $(2^n-1)/2^{2n}$. We also have that:
    \begin{itemize}
    \item $f(I_n)=I_{n+1}$ for $n\leq 0$;
    \item $f(I_n)=I_{n,1}^1$ for $n\geq 1$;
    \item $f(I_{n,j}^l)=I_{n,j}^{l+1}$ if $1\leq l \leq 4n-1 \text{ and } 1\leq j\leq 2^n$;
    \item $f(I_{n,j}^l)=I_{n,j+1}^1$ if $l=4n$ and $1\leq j \leq 2^n-1$;
    \item $f(I_{n,j}^l)=I_{n+1}$ if $l=4n$ and $j=2^n$.
\end{itemize}
 \Cref{figura} illustrates how $f$ goes from $I_0$ to $I_2$ and which measures are considered on which subintervals.
    
    We have that $(X,\mathcal{B},\mu,f)$ is an invertible measurable system, and since $\mu(f^{-1}(B))\leq 2\mu(B)$ and $\mu(f(B))\leq 2 \mu(B)$ for every $B\in \mathcal{B},$ both $T_f$ and $(T_f)^{-1}=T_{f^{-1}}$ are continuous on $\mathcal{X}$, see \cite{bayart2018topological} or \cite{singh1993composition}. Let us show that $f$ satisfies (MRC), but not (KRC).

    First, let $A=I_0=\{0\}\times [0,1]$. Let $N\geq 0$, and let $n_N\geq 0$ be such that $f^{n_N}(I_0)=I_N.$ For every $n_N+1 \leq n \leq n_N+2^N\cdot 4N,$ if $n_N+1+(j-1)\cdot 4N \leq n \leq n_N+j \cdot 4N$, $1\leq j\leq 2^N$, let 
    \[
    B_n=\{0\}\times \Big( [0,1]\setminus \Big[\frac{j-1}{2^N},\frac{j}{2^N}\Big]\Big)\subseteq I_0,
    \]
    and let \[
    B_{n_N}=\{0\}\times \Big([0,1]\setminus \Big[0,\frac{1}{2^N}\Big]\Big)\subseteq I_0.
    \] 
    By construction of $f$ and $\mu$ we have that 
    \[
    \mu(A\setminus B_n)=\mu(I_0\setminus B_n) =\frac{1}{2^N}\hspace*{0.5cm}\text{ and }\hspace*{0.5cm}\mu(f^n(B_n))=\frac{1}{2^N}\Big(1-\frac{1}{2^N}  \Big)\leq \frac{1}{2^N}
    \]
    for every $n_N\leq n \leq n_N+2^N\cdot 4N$.
    
   Since $n_0=0$ and $n_N+2^N\cdot 4N+1 = n_{N+1}$, $N\geq 0$, we have defined $B_n\subseteq A$ for all $n\geq 0$, with
   \[
   \mu(A\setminus B_n)\to 0 \hspace*{0.5cm}\text{ and }\hspace*{0.5cm} \mu(f^n( B_n))\to 0.
   \]

By the definition of the measures $\mu_n$ on $I_n$ for $n< 0$, it is clear that we also have that 
\[
    \mu(f^{-n}( B_n))\to 0.
   \]

   Next, if $A\subseteq I_0=\{0\}\times [0,1]$ is an arbitrary measurable set, then it suffices to consider the sets 
\[
A\cap B_n, 
\]
where the $B_n$ have been constructed above. Thus (MRC) holds for all measurable sets $A\subseteq I_0$. 
      
    For the general case, let $A\subseteq X$ be of finite measure. In order to prove (MRC) also in this case it suffices to show that, for every $\varepsilon>0$, there is a measurable set $A'\subseteq A$ with $\mu(A\setminus A')<\varepsilon$ and  measurable sets $B_n\subseteq A'$, $n\geq 1$, such that
\begin{equation}\label{eq: measureA'}
		\mu(A'\setminus B_n)\to 0,\hspace*{0.5cm} \mu(f^{-n}(B_n))\to 0\hspace*{0.5cm} \text{and}\hspace*{0.5cm} \mu(f^{n}(B_n))\to 0.
\end{equation}	
Because then, taking $\varepsilon =\frac{1}{k}$, we can find an increasing sequence $(n_k)_{k\geq 1}$ of positive integers and measurable sets $B_n\subseteq A$, $n\geq 1$, such that, whenever $n_k\leq n <n_{k+1}$, $k\geq 1$,	
\[
		\mu(A\setminus B_n)<\frac{1}{k},\hspace*{0.5cm} \mu(f^{-n}(B_n))<\frac{1}{k}\hspace*{0.5cm} \text{and}\hspace*{0.5cm} \mu(f^{n}(B_n))<\frac{1}{k}.
\]
Hence (MRC) also holds for $A$.
		
Thus let $A\subseteq X$ be of finite measure and $\varepsilon>0$.	Then there exists a finite set
\[
J\subseteq \mathbb{Z}\cup\{\alpha_{n,j}^l : n\in\mathbb{N}, 1\leq l\leq 4n, 1\leq j\leq 2^n\}
\]
such that $A':= A\cap \bigcup_{\alpha\in J} (\{\alpha\}\times [0,1])$ satisfies
\[
\mu(A\setminus A')<\varepsilon.
\]
In the sequel write, for $\alpha\in J$,
\[
A_\alpha= A\cap (\{\alpha\}\times [0,1]),
\]
so that $A'=\bigcup_{\alpha\in J} A_\alpha$.

Now, for every $\alpha\in J$ there is some $m_\alpha\in\mathbb{Z}$ such that
\[
f^{-m_\alpha}(\{\alpha\}\times [0,1]) = \{0\}\times [0,1].
\]
Moreover, by construction of the measure, we have that for any $\alpha\in J$ and any sequence $(C_n)_n$ of measurable subsets of $\{0\}\times [0,1]$,
\begin{equation}\label{eq: measureC}
\mu(C_n)\to 0 \implies \mu(f^{m_\alpha}(C_n))\to 0.
\end{equation}

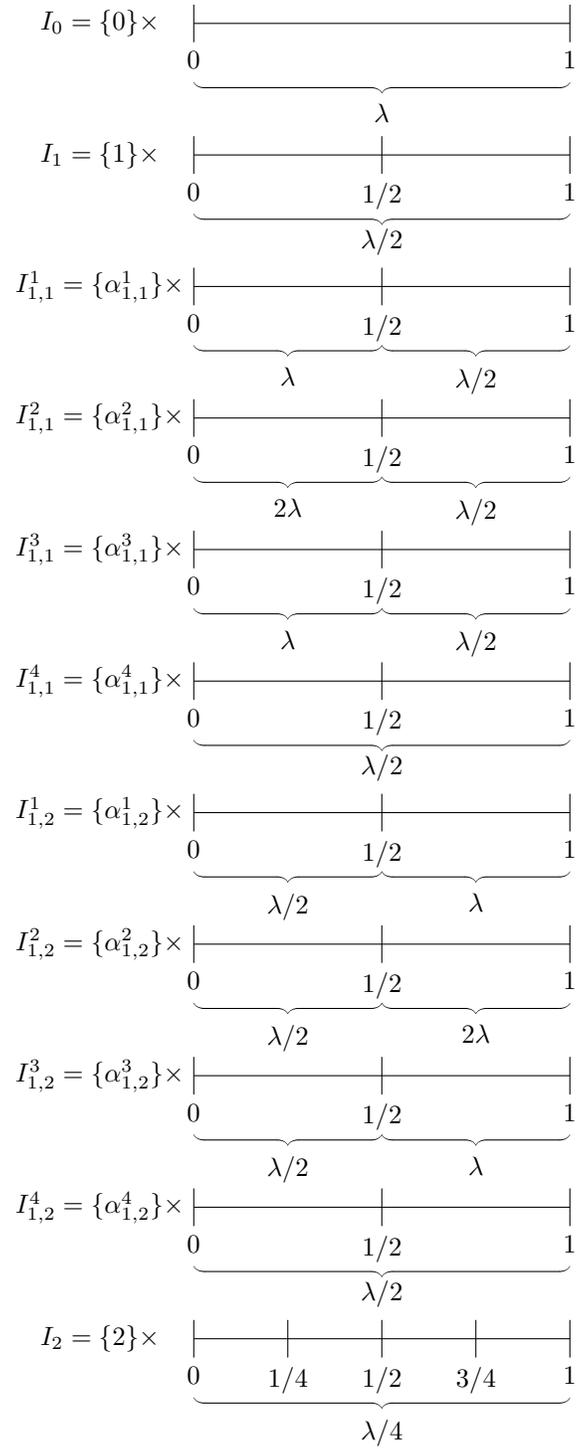
\begin{figure}
    \centering
	\begin{tikzpicture}[scale=5]
    \draw (0,0.35)-- (1,0.35);
    \foreach \x in {0,1} {
        \draw (\x,0.4) -- (\x,0.3) node[below] {\x};
    }
    \draw[B] (0,0.2) -- node[below=2mm] {$\lambda$} (1,0.2);
    \draw (0,0)-- (1,0); 
    \foreach \x in {0,1,1/2} {
        \draw (\x,0.05) -- (\x,-0.05) node[below] {\x};
    }
    \draw[B] (0,-0.15) --  node[below=1mm] {$\lambda/2$} (1,-0.15);
    
    \draw (0,-0.35)-- (1,-0.35); 
    \foreach \x in {0,1,1/2} {
        \draw (\x,-0.30) -- (\x,-0.40) node[below] {\x};
    }
    \draw[B] (0,-0.5) -- node[below=2mm] {$\lambda$} (1/2,-0.5);
    \draw[B] (1/2,-0.5) -- node[below=2mm] {$\lambda/2$} (1,-0.5);
    
    \draw (0,-0.7)-- (1,-0.7); 
    \foreach \x in {0,1,1/2} {
        \draw (\x,-0.65) -- (\x,-0.75) node[below] {\x};
    }
    \draw[B] (0,-0.85) -- node[below=2mm] {$2\lambda$} (1/2,-0.85);
    \draw[B] (1/2,-0.85) -- node[below=2mm] {$\lambda/2$} (1,-0.85);
    
    \draw (0,-1.05)-- (1,-1.05); 
    \foreach \x in {0,1,1/2} {
        \draw (\x,-1) -- (\x,-1.1) node[below] {\x};
    }
    \draw[B] (0,-1.2) -- node[below=2mm] {$\lambda$} (1/2,-1.2);
    \draw[B] (1/2,-1.2) -- node[below=2mm] {$\lambda/2$} (1,-1.2);
    
    \draw (0,-1.4)-- (1,-1.4); 
    \foreach \x in {0,1,1/2} {
        \draw (\x,-1.35) -- (\x,-1.45) node[below] {\x};
    }
    \draw[B] (0,-1.55) --  node[below=1mm] {$\lambda/2$} (1,-1.55);
    
    \draw (0,-1.75)-- (1,-1.75); 
    \foreach \x in {0,1,1/2} {
        \draw (\x,-1.7) -- (\x,-1.8) node[below] {\x};
    }
    \draw[B] (0,-1.9) -- node[below=2mm] {$\lambda/2$} (1/2,-1.9);
    \draw[B] (1/2,-1.9) -- node[below=2mm] {$\lambda$} (1,-1.9);

    \draw (0,-2.1)-- (1,-2.1); 
    \foreach \x in {0,1,1/2} {
        \draw (\x,-2.05) -- (\x,-2.15) node[below] {\x};
    }
    \draw[B] (0,-2.25) -- node[below=2mm] {$\lambda/2$} (1/2,-2.25);
    \draw[B] (1/2,-2.25) -- node[below=2mm] {$2\lambda$} (1,-2.25);
    
    \draw (0,-2.45)-- (1,-2.45); 
    \foreach \x in {0,1,1/2} {
        \draw (\x,-2.4) -- (\x,-2.5) node[below] {\x};
    }
    \draw[B] (0,-2.6) -- node[below=2mm] {$\lambda/2$} (1/2,-2.6);
    \draw[B] (1/2,-2.6) -- node[below=2mm] {$\lambda$} (1,-2.6);
	    
    \draw (0,-2.8)-- (1,-2.8); 
    \foreach \x in {0,1,1/2} {
        \draw (\x,-2.75) -- (\x,-2.85) node[below] {\x};
    }
    \draw[B] (0,-2.95) --  node[below=1mm] {$\lambda/2$} (1,-2.95);
    
    \draw (0,-3.15)-- (1,-3.15); 
    \foreach \x in {0,1,1/2,1/4,3/4} {
        \draw (\x,-3.1) -- (\x,-3.2) node[below] {\x};
    }
    \draw[B] (0,-3.3) --  node[below=2mm] {$\lambda/4$} (1,-3.3);
    
    \draw (-0.25,0.35) node {$I_0=\{0\}\times $};
    \draw (-0.25,0) node {$I_1=\{1\}\times $};
    \draw (-0.25,-0.35) node {$I_{1,1}^1=\{\alpha_{1,1}^1\}\times $};
    \draw (-0.25,-0.7) node {$I_{1,1}^2=\{\alpha_{1,1}^2\}\times $};
    \draw (-0.25,-1.05) node {$I_{1,1}^3=\{\alpha_{1,1}^3\}\times $};
    \draw (-0.25,-1.4) node {$I_{1,1}^4=\{\alpha_{1,1}^4\}\times $};
    \draw (-0.25,-1.75) node {$I_{1,2}^1=\{\alpha_{1,2}^1\}\times $};
    \draw (-0.25,-2.1) node {$I_{1,2}^2=\{\alpha_{1,2}^2\}\times $};
    \draw (-0.25,-2.45) node {$I_{1,2}^3=\{\alpha_{1,2}^3\}\times $};
    \draw (-0.25,-2.8) node {$I_{1,2}^4=\{\alpha_{1,2}^4\}\times $};
    \draw (-0.25,-3.15) node {$I_2=\{2\}\times $};
    \end{tikzpicture}    
    \caption{Construction of $f$ from $I_0$ to $I_2$}
    \label{figura}
    \end{figure}

Since (MRC) holds for measurable subsets of $I_0$, we find for any $\alpha\in J$ measurable sets 
\[
B_n^\alpha \subseteq f^{-m_\alpha}(A_\alpha)
\]
such that
\[
		\mu(f^{-m_\alpha}(A_\alpha)\setminus B_n^\alpha)\to 0\hspace*{0.5cm}\text{and}\hspace*{0.5cm} \mu(f^{n}(B_n^\alpha))\to 0,
\]
hence also, for $n\geq \max\{0,-m_\alpha\}$,
\[
		\mu(f^{-m_\alpha}(A_\alpha)\setminus B_{m_\alpha+n}^\alpha)\to 0\hspace*{0.5cm}\text{and}\hspace*{0.5cm} \mu(f^{m_\alpha+n}(B_{m_\alpha+n}^\alpha))\to 0,
\]

It then follows from \eqref{eq: measureC}  that, for any $\alpha\in J$,
\[
\mu(A_\alpha\setminus f^{m_\alpha}(B_{m_\alpha+n}^\alpha))\to 0.
\]

Finally, for $n\geq \max\{0,-m_\alpha:\alpha\in J\}$, let 
\[
B_n:=\bigcup_{\alpha \in J} f^{m_\alpha}(B_{m_\alpha+n}^\alpha)\subseteq A'.
\]
Then we have that
\[
\mu(A'\setminus B_n)=\sum_{\alpha\in J} \mu(A_\alpha \setminus  f^{m_\alpha}(B_{m_\alpha+n}^\alpha))\to 0
\]
and 
\[
 \mu(f^{n}(B_n))=\sum_{\alpha\in J} \mu(f^{m_\alpha+n}(B_{m_\alpha+n}^\alpha))\to 0.
\]
By the definition of the measures $\mu_n$ on $I_n$ for $n< 0$, it is clear that we also have that 
\[
    \mu(f^{-n}( B_n))\to 0.
\]
This shows \eqref{eq: measureA'}.

Altogether, $f$ satisfies (MRC).
    
    Let us now show that $f$ does not satisfy (KRC). Let $A=I_0$ and $\varepsilon=1/4$. Let $B=\{0\}\times C \subseteq I_0$, with $C\subseteq [0,1]$ measurable, be such that $\mu(A\setminus B)<\varepsilon$, that is, $\lambda ([0,1]\setminus C)<\varepsilon.$ We must have that 
    \[
    \lambda\Big(C\cap \Big[0,\frac{1}{2}\Big]\Big)\geq \frac{\lambda(C)}{2}\hspace*{0.5cm} \text{ or } \hspace*{0.5cm} \lambda\Big(C\cap \Big[\frac{1}{2},1\Big]\Big)\geq \frac{\lambda(C)}{2}.
    \]
    More generally, we must have that, for every $n\geq 1$, there exists some $1\leq j_n\leq 2^n$ such that
    \[
    \lambda\Big(C\cap \Big[\frac{j_n-1}{2^n},\frac{j_n}{2^n}\Big]\Big)\geq \frac{\lambda(C)}{2^n}.
    \]
    For $n\geq 1,$ let $K_n$ be such that $f^{K_n}(I_0)=I_{n,j_n}^{2n}.$ Then
    \[
    f^{K_n}(B)=\{\alpha_{n,j_n}^{2n}\}\times C\supseteq \{\alpha_{n,j_n}^{2n}\}\times \Big(C\cap \Big[\frac{j_n-1}{2^n},\frac{j_n}{2^n}\Big]\Big)
    \]
    and hence
    \[
    \mu(f^{K_n}(B)) \geq 2^n \cdot \lambda\Big(C\cap \Big[\frac{j_n-1}{2^n},\frac{j_n}{2^n}\Big]\Big)\geq \lambda(C).
    \]
    Thus we have that \[\mu(f^k(B))\geq \lambda(C)\geq 1-\varepsilon = \frac{3}{4}>\varepsilon\]
    for infinitely many $k$, implying that $f$ cannot satisfy (KRC).
\end{proof}

\begin{remark}\label{rem: Grivaux}
(a) Recall that an operator is called \textit{chaotic} if it is hypercyclic and has a dense set of periodic points. By \cite[Theorem~3.2]{darji2021}, the previous composition operator is not chaotic. Note that our system is dissipative; see the following section.

(b) We claim that our example has the property that the only function $\varphi\in \mathcal{X}$ satisfying $T_f^{-n}\varphi\to 0$ is the zero function ($\mu$-almost everywhere). Indeed, suppose that $\varphi$ does not vanish on a set of strictly positive measure. By considering separately the real and imaginary part of $\varphi$, it is enough to consider real functions, and by considering separately the negative and positive part, it is enough to look at positive real functions $\varphi$.
Doing so, there exist $\alpha$, a measurable set $B\subseteq \{\alpha\}\times [0,1]$ of stricly positive measure and $\delta>0$ such that $\varphi\geq \delta \chi_B$. This implies that $\varphi \circ f^{-n} \geq \delta \chi_B \circ f^{-n}$, so that $T_f^{-n} \varphi \geq \delta \chi_{f^{n}(B)}$. By taking the norm, we have $\|T^{-n}_f \varphi\|^p \geq \delta^p \mu(f^{n}(B))$. Using the same argument as the final part of the proof of \cref{principal}, one can show that no $B\subseteq \{\alpha\}\times [0,1]$ of strictly positive measure satisfies $\mu (f^{n}(B))\to 0$, proving our claim.

(c) By regarding $f^{-1}$, we also have an example of an invertible measurable system for which $T_f$ is mixing but where only the zero function $\varphi$ satisfies that $T_f^{n}\varphi\to 0$.
Grivaux constructed in \cite{grivaux2005hypercyclic} an example of a mixing operator $T$ on $\ell^p(\mathbb{N})$, $1\leq p<\infty$, that does not satisfy Kitai's Criterion by showing that $x=0$ is the only vector satisfying $T^nx\to 0$. For this she uses a deep operator theoretic result.

(d) Menet and Papathanasiou \cite{menetpapathanasiou24} have recently studied generalized shifts induced by sequences of operators. In this context they have also found mixing operators that do not satisfy Kitai's Criterion, see \cite[Section 4.1]{menetpapathanasiou24}. They point out that there are four (weaker and stronger) versions of Kitai's Criterion; they show that two of these criteria are equivalent, while the others are not. The example considered in (c) above, just like Grivaux's example, does not satisfy any version of Kitai's Criterion.
\end{remark}

The investigations in this section lead one to wonder whether, in the characterization of hypercyclic composition operators, \cref{th: characterization hc}, and for invertible measurable systems $(X,\mathcal{B},\mu,f)$, one might replace the sequence of measurable sets $B_k$ in condition (HRC) by a single set $B$. More precisely, one might ask whether (HRC) can be replaced by the following \textit{Gethner-Shapiro Runaway Condition}:

\begin{enumerate}
    \item[(GRC)] \textit{For every $A \in \mathcal{B}$ with finite measure and every $\varepsilon >0$ there exist a measurable set $B\subseteq A$ and an increasing sequence $(n_k)_{k\geq 1}$ of positive integers such that
    \[
		\mu(A\setminus B)<\varepsilon,\hspace*{0.5cm} \mu(f^{-n_k}(B))\to 0 \hspace*{0.5cm} \text{and} \hspace*{0.5cm} \mu(f^{n_k}(B))\to 0.
		\]
  }
\end{enumerate}

The name of this condition refers to the Gethner-Shapiro Criterion, which is exactly Kitai's Criterion where the full sequence $(n)_n$ is replaced by an increasing sequence $(n_k)_k$ of positive integers, see \cite{gethner1987} and \cite[Theorem 3.10]{grosse2011linear}. This criterion implies that the operator is weakly mixing.

Now, Kalmes \cite[p. 1603 and Example 3.19]{Kalmes2007} has shown that one cannot always choose all the sets $B_k$ in (HRC) to be $A$. But, maybe surprisingly, the answer to the question above is positive. Note that if $f$ is bijective and bimeasurable then $f^{-1}(\mathcal{B})=_{\text{ess}}\mathcal{B}$.

\begin{theorem}\label{thm: Gethner-Shapiro characterization}
Let $(X,\mathcal{B},\mu,f)$ be an invertible measurable system. Suppose that $\mathcal{X}$ satisfies \emph{(H1)}, \emph{(H2)}, \emph{(C1)} and \emph{(C2)}. The following are equivalent:
\begin{itemize}
    \item[(a)] $T_f$ is hypercyclic;
    \item[(b)] $T_f$ is weakly mixing;
    \item[(c)] $f$ satisfies \emph{(HRC)};
    \item[(d)] $f$ satisfies \emph{(GRC)}.
\end{itemize}
\end{theorem}

\begin{proof}
By \cref{th: characterization hc} it suffices to show that (b) implies (GRC). Now, by Peris \cite{peris2001}, see also \cite[Theorem 3.22]{grosse2011linear}, every weakly mixing operator satisfies the Gethner-Shapiro Criterion. Then (GRC) can be deduced exactly as in the proof of \cref{thm: KitaiNec}.
\end{proof}

\section{Dissipative Systems}\label{sec: dissipative systems}

The bilateral (respectively unilateral) backward shift $B$ on a sequence space $\mathcal{X}\subseteq \mathbb{K}^\mathbb{Z}$ (respectively $\mathcal{X}\subseteq \mathbb{K}^\mathbb{N}$, $\mathbb{K}=\mathbb{R}$ or $\mathbb{C}$) is a natural and important example of a composition operator. It is defined by
\[
B(x_n)_n = (x_{n+1})_n,
\]
so that $B=T_f$ with $f(n)=n+1$. In this context, the canonical unit sequences will be denoted by $e_n$. For more details regarding properties of $B$, see \cite[Section~4.1]{grosse2011linear}.

However, the next example shows that not every backward shift can be treated in the framework of \cref{sec: composition operators}.

\begin{example}\label{ex:seqsp}
Let $\mathcal{X}$ be the sequence space over $\mathbb{N}$ given by
\[
\mathcal{X}=\Big\{ (x_n)_{n\geq 1} : \sum_{n=1}^\infty \frac{x_n}{n} \ \text{ converges}\Big\},
\]
with norm $\|(x_n)_n\|=\sup_{N\geq 1} \big| \sum_{n=1}^N \frac{x_n}{n}\big|$. Then $(\mathcal{X},\|\cdot\|)$ is a Banach space for which the unit sequences $e_n$, $n\geq 1$, form a basis. It is not difficult to see, using Abel partial summation, that the backward shift $B$ is an operator on $\mathcal{X}$. But there is no measure $\mu$ on $\mathcal{P}(\mathbb{N})$ that satisfies (H1) and (C1). Indeed, since $((-1)^n\sqrt{n})_n\in \mathcal{X}$, (C1) would imply, via \cref{lemma}, that there is some $m\geq 1$ so that
\[
\mu\Big(\Big\{n\in \mathbb{N}: \frac{1}{m} \sqrt{n}\geq 1\Big\}\Big)<1.
\]
But, then, (H1) would imply that $\chi_{[m^2,\infty)}\in\mathcal{X}$, a contradiction. 

Note also that $B$ is hypercyclic (see \cite[Theorem 4.3]{grosse2011linear}).
\end{example}

Thus there are natural backward shifts whose hypercyclicity cannot be proved by the results in \cref{sec: composition operators}. On the other hand, backward shifts are dissipative systems. This motivates us to consider a framework on dissipative systems that is wider than in \cref{sec: composition operators}.

Recall that a measurable system $(X,\mathcal{B},\mu,f)$ or the composition operator $T_f:\mathcal{X}\to\mathcal{X}$ is called \textit{dissipative} if there exists $W\in \mathcal{B}$ such that
\begin{enumerate}
    \item[(i)] $f^m(W)\in \mathcal{B}$, for all $m\in \mathbb{Z}$;
    \item[(ii)] the $f^m(W)$, $m\in \mathbb{Z}$, are pairwise disjoint;
    \item[(iii)] $X=\bigcup_{m\in \mathbb{Z}} f^m(W)$.
\end{enumerate}
Such a $W$ is called a \textit{wandering set} or a \textit{generator} of the system. We emphasize that we do not exclude that $\mu(W)=0$ or $\infty$.

For example, the measurable system of \cref{principal} is dissipative with wandering set $W=\{0\}\times [0,1]$.

For every $m\in \mathbb{Z}$, let $W_m=f^m(W)$. We will consider the following local (i.e. depending on $W$) versions of (H1), (H2), (C1) and (C2), respectively, on $\mathcal{X}$:
\begin{enumerate}
    \item[(LH1)] For any $A\in \mathcal{B}(W_m)$, $m\in \mathbb{Z}$, with finite measure, $\chi_A\in\mathcal{X};$
    \item[(LH2)] The set of simple functions of the form $\sum_{j=1}^M a_j \chi_{A_j}$ is dense in $\mathcal{X}$, where for every $j$, $\mu(A_j)<\infty$ and $A_j\in \mathcal{B}(W_m)$ for some $m\in \mathbb{Z}$;
    \item[(LC1)] The family of maps
    \[
        I_m:\mathcal{X}\to L^0_\mu(W_m),\ \varphi \mapsto \varphi |_{W_m},
    \]
    $m\in \mathbb{Z}$, is equicontinuous;
    \item[(LC2)] The family of maps		
		\[
        J_A:\mathcal{B}(A)\to \mathcal{X},\ C\mapsto \chi_C,
    \]
		$A\in \mathcal{B}(W_m)$ for some $m\in \mathbb{Z}$ and of finite measure, is equicontinuous.
\end{enumerate}
We interpret again conditions (LH2) and (LC2) as implying (LH1).

\begin{remark}
The conditions (H1), (C1) and (C2) imply the corresponding local conditions (use \cref{lem-cont} for (C2)), while (H2)$\&$(C2) implies (LH2). Thus the set of four local conditions is weaker than the set (H1), (H2), (C1) and (C2).

When we take, in \cref{ex:seqsp}, the measure given by $\mu(\{n\})=\frac{1}{n}$, $n\in\mathbb{N}$, and $W=\{1\}$ then we see that (LH1)$\&$(LC1) holds (see the comment after \cref{thm:bilshift} below), while (H1)$\&$(C1) does not. When we take, in the same example, for $\mu$ a finite measure and $W=\{1\}$ then (LH2)$\&$(LC2) holds while (H1) does not, so that (H2) and (C2) are not even well defined.
\end{remark}

The following theorem shows that, for dissipative systems, and under the weaker conditions (LH1) and (LC1), recurrence implies hypercyclicity, so that these two notions are equivalent in this context.

\begin{theorem}\label{th: recurrent}
     Let $(X,\mathcal{B},\mu,f)$ be a dissipative system, with $\mathcal{X}$ satisfying \emph{(LH1)} and \emph{(LC1)}. If $T_f$ is recurrent, then $f^{-1}(\mathcal{B})=_{\emph{\text{ess}}}\mathcal{B}$ and f satisfies \emph{(HRC)}.
\end{theorem}

\begin{proof}
    Since $T_f$ is recurrent, it has dense range, so that $f^{-1}(\mathcal{B})=_{\text{ess}}\mathcal{B}$ by \cref{Whitley}(a).

    Now let $A\in \mathcal{B}$ with finite measure and $0<\varepsilon<1$. It suffices again to find a measurable set $B\subseteq A$ and some $k\geq 1$ such that \eqref{eq: HRC1} holds.     
    Since
    \[
		\mu(A)=\sum_{m\in \mathbb{Z}} \mu(A\cap W_m),
		\]
    there exists $N\geq 0$ such that 
    \[
		\sum_{|m|> N} \mu(A\cap W_m)<\varepsilon.
		\]
    Thus we can suppose that 
    \begin{equation}\label{eq: A subset W_n}
        A\subseteq \bigcup_{|m|\leq N}W_m.
    \end{equation}

    Using (LH1) and the equicontinuity of the family $(I_m)_{m\in \mathbb{Z}}$ by (LC1), let $\delta>0$ be such that
    \begin{equation}\label{eq: equicontinuity}
        \|\psi-2\chi_A\|<\delta \implies d(\psi |_{W_m},2\chi_A |_{W_m})<\frac{\varepsilon}{2(2N+1)} \hspace*{0.5cm} \forall m\in \mathbb{Z}.
    \end{equation}

    Since $T_f$ is recurrent, there exist $\varphi\in \mathcal{X}$ and $k\geq 1$ such that
    \begin{equation}\label{eq: recurrence}
        \|\varphi-2\chi_A\|<\delta,\hspace*{0.5cm} \|\varphi \circ f^k-2\chi_A\|<\delta \hspace*{0.5cm}\text{and}\hspace*{0.5cm} W_m\cap (W_{l-k}\cup W_{l+k})=\varnothing
    \end{equation}
    for all $|m|,|l|\leq N$.

    Let \[C=\{x\in X: |\varphi(x)-2|<1\},\]
    and define \[B=C\cap f^{-k}(C)\cap A.\]
    
    First, if $x\in (A\cap W_m)\setminus C$, $m\in \mathbb{Z}$, then $|\varphi (x)-2\chi_A(x)|=|\varphi(x)-2|\geq 1$,  so that \labelcref{eq: recurrence}, \labelcref{eq: equicontinuity} and \cref{lemma} imply
    \[\mu((A\cap W_m)\setminus C)<\frac{\varepsilon}{2(2N+1)}.\]
    Hence, by \labelcref{eq: A subset W_n},  $\mu(A\setminus C)<\varepsilon/2$.

    Next,  if $x\in (A\cap W_m)\setminus f^{-k}(C)$, $m\in \mathbb{Z}$, then $|\varphi \circ f^k(x)-2\chi_A(x)|=|\varphi\circ f^k(x)-2|\geq 1$,  so that \labelcref{eq: recurrence}, \labelcref{eq: equicontinuity} and \cref{lemma} also imply that 
    \[\mu((A\cap W_m)\setminus f^{-k}(C))<\frac{\varepsilon}{2(2N+1)}.\]
    Hence, by \labelcref{eq: A subset W_n},  $\mu(A\setminus f^{-k}(C))<\varepsilon/2$.

    Altogether we deduce that $\mu(A\setminus B)<\varepsilon$.

    Now, let $x\in f^{-k}(C)\cap W_{m-k}$, $|m|\leq N$. By \labelcref{eq: recurrence} and \labelcref{eq: A subset W_n}, $x\notin A$, so that $|\varphi \circ f^k(x)-2\chi_A(x)|=|\varphi \circ f^k(x)|> 1$. By \labelcref{eq: recurrence}, \labelcref{eq: equicontinuity} and \cref{lemma} we get that 
    \[\mu(f^{-k}(C)\cap W_{m-k})<\frac{\varepsilon}{2(2N+1)}, \hspace*{0.5cm} |m|\leq N\]
    and hence, by \labelcref{eq: A subset W_n},
    \[\mu(f^{-k}(B))\leq \mu(f^{-k}(C\cap A))\leq \sum_{|m|\leq N} \mu(f^{-k}(C)\cap W_{m-k})<\varepsilon.\]

    Finally, let $x\in C\cap W_{m+k}$, $|m|\leq N$. By \labelcref{eq: recurrence} and \labelcref{eq: A subset W_n}, $x\notin A$, so that $|\varphi (x)-2\chi_A(x)|=|\varphi(x) |> 1$. By \labelcref{eq: recurrence}, \labelcref{eq: equicontinuity} and \cref{lemma} we see that 
    \[\mu(C\cap W_{m+k})<\frac{\varepsilon}{2(2N+1)}, \hspace*{0.5cm} |m|\leq N\]
    and therefore, by \labelcref{eq: A subset W_n},
    \[\mu^*(C\cap f^{k}(A))\leq \sum_{|m|\leq N} \mu(C\cap W_{m+k})<\varepsilon. \]
    Since $f^k(B)\subseteq C\cap f^k(A)$, it follows that $\mu^*(f^k(B))<\varepsilon$.
\end{proof}

For the analogue of \cref{th: ssc implies hc} we need a local version of \cref{lem-cont}.

\begin{lemma}\label{lem-contloc}
Condition \emph{(LC2)} holds if, and only if, for any $N\geq 1$ and for any $\varepsilon>0$ there is some $\delta>0$ such that, if $A,B\in\mathcal{B}$ satisfy $A\subseteq \bigcup_{m\in I} W_m$ and $B\subseteq \bigcup_{m\in J} W_m$ with $\text{\emph{card}}(I)$, $\text{\emph{card}}(J)\leq N$, then 
\[
\mu(A\Delta B)<\delta \Longrightarrow \|\chi_A-\chi_B\|<\varepsilon.
\]
\end{lemma}

Indeed, it suffices to show that (LC2) implies the stated condition for $B=\varnothing$. The case $N=1$ holds with (LC2), the general case requires an application of the triangle inequality.

We now have the announced analogue.

\begin{theorem}\label{th: dissipative src implies weak mixing}
   Let $(X,\mathcal{B},\mu,f)$ be a dissipative system, with $\mathcal{X}$ satisfying conditions \emph{(LH2)} and \emph{(LC2)}. Assume that $f^{-1}(\mathcal{B})=_{\emph{\text{ess}}}\mathcal{B}$ and that $f$ satisfies \emph{(HRC)}. Then $T_f$ is weakly mixing.
\end{theorem}

\begin{proof} We follow verbatim the proof of \cref{th: ssc implies hc}. By (LH2), the set $A$ may belong to $\bigcup_{|m|\leq N} W_m$ for some $N\geq 0$. Then, by (LC2) via \cref{lem-contloc}, equation \labelcref{eq: continuity of J} holds whenever $E$ and $F$ belong to sets $\bigcup_{|m|\leq N} W_{m+l}$ with (possibly different) $l\in\mathbb{Z}$. A careful reading of the proof shows that \labelcref{eq: continuity of J} is only applied to sets of this form (with $l\in\{-k,0,k\}$), where we note that one may choose $C$ to belong to $\bigcup_{|m|\leq N} W_{m+k}$.
\end{proof}

With the aim of approaching the characterizations obtained in this section to the ones known for backward shift operators on sequence spaces, see \cite[Theorem~4.12]{grosse2011linear}, we introduce here two conditions that are specific to the dissipative case. The \textit{Hypercyclic Shift-like Condition} (borrowing the terminology from \cite{daniello2022shiftlike}) reads as follows.  
    \begin{itemize}
        \item[(HSC)]\textit{There exists an increasing sequence of integers $(n_k)_{k\geq 1}$ such that for all $m\in \mathbb{Z}$, for all measurable sets $A\subseteq W_m$ with finite measure, there exists a sequence of measurable sets $B_k\subseteq A$, $k\geq 1$, such that}
				\[
				\mu(A\setminus B_k)\to 0,\hspace*{0.5cm} \mu(f^{-n_k}(B_k))\to 0 \hspace*{0.5cm}\text{and}\hspace*{0.5cm} \mu^*(f^{n_k}(B_k))\to 0.
				\]  
    \end{itemize}
Analogously, we have a \textit{Mixing Shift-like Condition}:
    \begin{itemize}
        \item[(MSC)] \textit{For all $m\in \mathbb{Z}$, for all measurable sets $A\subseteq W_m$ with finite measure, there exists a sequence of measurable sets $B_n\subseteq A$, $n\geq 1$, such that}
				\[
				\mu(A\setminus B_n)\to 0,\hspace*{0.5cm} \mu(f^{-n}(B_n))\to 0 \hspace*{0.5cm}\text{and}\hspace*{0.5cm} \mu^*(f^{n}(B_n))\to 0.
				\]  
    \end{itemize}
		
The equivalence to the conditions (HRC) and (MRC), respectively, is obvious if one takes into account that $\mu(A\setminus \bigcup_{|m|\leq N} W_m) \to 0$; for (HRC) one also has to note \cref{rem: BDP}(c).

Thus by \cref{th: recurrent,th: dissipative src implies weak mixing}, we get the following characterization for dissipative systems.

\begin{theorem}\label{th: dissipative}
      Let $(X,\mathcal{B},\mu,f)$ be a dissipative system, with $\mathcal{X}$ satisfying \emph{(LH1)},  \emph{(LH2)}, \emph{(LC1)} and \emph{(LC2)}. The following are equivalent:
    \begin{enumerate}
        \item[\emph{(a)}] $T_f$ is recurrent;
        \item[\emph{(b)}] $T_f$ is hypercyclic;
        \item[\emph{(c)}] $T_f$ is weakly mixing;
        \item[\emph{(d)}] $f^{-1}(\mathcal{B})=_{\emph{\text{ess}}}\mathcal{B}$ and $f$ satisfies \emph{(HRC)};
        \item[\emph{(e)}] $f^{-1}(\mathcal{B})=_{\emph{\text{ess}}}\mathcal{B}$ and $f$ satisfies \emph{(HSC)}.
    \end{enumerate}
\end{theorem}

This result not only generalizes \cite[Theorem~2.2]{daniello2024} from $L^p(X,\mathcal{B},\mu)$ to general spaces of measurable functions, it improves it because we do not require $f$ to be invertible.

\begin{remark}
    If we drop the dissipativity hypothesis, $T_f$ being recurrent does not necessarily imply that $T_f$ is hypercyclic. For example, if we take $X=\mathbb{R}$ with the Lebesgue measure and $f:X\to X$ to be the identity, then $T_f:L^2(X)\to L^2(X)$ is recurrent but it is not hypercyclic.
\end{remark}

We again have analogous results to \cref{th: dissipative,th: dissipative src implies weak mixing,th: recurrent} for mixing operators in the dissipative context. In particular, we have the following characterization:

\begin{theorem}\label{t-dissipative-mix}
      Let $(X,\mathcal{B},\mu,f)$ be a dissipative system, with $\mathcal{X}$ satisfying \emph{(LH1)},  \emph{(LH2)}, \emph{(LC1)} and \emph{(LC2)}. The following are equivalent:
    \begin{enumerate}
        \item[\emph{(a)}] $T_f$ is mixing;
        \item[\emph{(b)}] $f^{-1}(\mathcal{B})=_{\emph{\text{ess}}}\mathcal{B}$ and $f$ satisfies \emph{(MRC)};
        \item[\emph{(c)}] $f^{-1}(\mathcal{B})=_{\emph{\text{ess}}}\mathcal{B}$ and $f$ satisfies \emph{(MSC)}.
    \end{enumerate}
\end{theorem}

We have also dissipative analogues of \cref{thm: KitaiNec,thm: Kitai,thm: kitai characterization,thm: Gethner-Shapiro characterization}. We first formulate the \textit{Kitai Shift-Like Condition} for an invertible dissipative system:
\begin{enumerate}
    \item[(KSC)] \textit{For all $m\in\mathbb{Z}$, for all measurable sets $A\subseteq W_m$ with finite measure and for every $\varepsilon>0$, there exists a measurable set $B\subseteq A$ such that
    \[
    \mu(A\setminus B)<\varepsilon, \hspace*{0.5cm} \mu(f^{-n}(B))\to 0 \hspace*{0.5cm} \text{and} \hspace*{0.5cm} \mu(f^n(B))\to 0.
    \] }
\end{enumerate}

Then we have in particular the following:

\begin{theorem}\label{thm: kitai characterization diss}
Let $(X,\mathcal{B},\mu,f)$ be an invertible dissipative system. Suppose that $\mathcal{X}$ satisfies \emph{(LH1)}, \emph{(LH2)}, \emph{(LC1)} and \emph{(LC2)}. The following are equivalent:
\begin{itemize}
    \item[(a)] $T_f$ satisfies Kitai's Criterion;
    \item[(b)] $f$ satisfies \emph{(KRC)};
    \item[(c)] $f$ satisfies \emph{(KSC)}.
\end{itemize}
\end{theorem}

We finally have the \textit{Gethner-Shapiro Shift-Like Condition} for an invertible dissipative system:
\begin{enumerate}
    \item[(GSC)] \textit{For all $m\in\mathbb{Z}$, for all measurable sets $A\subseteq W_m$ with finite measure and for every $\varepsilon>0$, there exists a measurable set $B\subseteq A$ and an increasing sequence $(n_k)_{k\geq 1}$ of positive integers such that
    \[
    \mu(A\setminus B)<\varepsilon, \hspace*{0.5cm} \mu(f^{-n_k}(B))\to 0 \hspace*{0.5cm} \text{and} \hspace*{0.5cm} \mu(f^{n_k}(B))\to 0.
    \] }
\end{enumerate}

Then we have:

\begin{theorem}\label{thm: Gethner characterization diss}
Let $(X,\mathcal{B},\mu,f)$ be an invertible dissipative system. Suppose that $\mathcal{X}$ satisfies \emph{(LH1)}, \emph{(LH2)}, \emph{(LC1)} and \emph{(LC2)}. The following are equivalent:
\begin{itemize}
    \item[(a)] $T_f$ is recurrent; 
    \item[(b)] $T_f$ is hypercyclic;
    \item[(c)] $T_f$ is weakly mixing;
    \item[(d)] $f$ satisfies \emph{(HRC)};
    \item[(e)] $f$ satisfies \emph{(GRC)};
    \item[(f)] $f$ satisfies \emph{(HSC)};
    \item[(g)] $f$ satisfies \emph{(GSC)}.
\end{itemize}
\end{theorem}

\section{An application to backward shift operators}\label{sec: backward shifts}

Our initial aim in \cref{sec: dissipative systems} had been to find a framework for dissipative systems that was large enough to cover the known results for backward shift operators on sequence spaces. Surprisingly, our results allow us to strictly enlarge the family of sequence spaces for which the classical characterization of hypercyclicity holds.

We first look at bilateral shifts.

\begin{theorem}\label{thm:bilshift}
Let $\mathcal{X} \subseteq \mathbb{K}^\mathbb{Z}$ be a Banach space of sequences in which $\text{\emph{span}}\{e_n:n\in\mathbb{Z}\}$ is dense and such that, for every $\varepsilon>0$, there is some $\delta>0$ such that
\begin{equation}\label{eq:LC1}
\|x\|<\delta \implies \forall n\in\mathbb{Z}, \ \min(|x_n|,\|e_n\|)<\varepsilon.
\end{equation}
Suppose that the backward shift $B$ is an operator on $\mathcal{X}$. 

\emph{(a)} $B$ is hypercyclic if, and only if, there is an increasing sequence of integers $(n_k)_{k\geq 1}$ such that, for all $j\in\mathbb{Z}$,
 \[
e_{j-n_k}\to 0 \hspace*{0.5cm} \text{and} \hspace*{0.5cm} e_{j+n_k}\to 0.
\]

\emph{(b)} $B$ is mixing if, and only if, for all $j\in\mathbb{Z}$,
 \[
e_{j-k}\to 0 \hspace*{0.5cm} \text{and} \hspace*{0.5cm} e_{j+k}\to 0.
\]
\end{theorem}

This result is a straightforward application of \cref{th: dissipative,t-dissipative-mix}, using conditions (HSC) and (MSC). It suffices to consider the measure $\mu$ on $\mathcal{P}(\mathbb{Z})$ given by $\mu(\{n\})=\|e_n\|$, $n\in \mathbb{Z}$, the transformation $f:\mathbb{Z}\to \mathbb{Z}$ given by $f(n)=n+1$, and $W=\{0\}$. Then $B=T_f:\mathcal{X}\to \mathcal{X}$ satisfies the hypotheses, where for (LC1) one calculates that $d(x|_{W_n},y|_{W_n})= \min(|x_n-y_n|,\mu(\{n\}))$ for $x=(x_n)_n$, $y=(y_n)_n\in\mathcal{X}$. But since the result is of independent interest, we give the (short) direct proof.

\begin{proof}
Sufficiency of the conditions holds whenever $\text{span}\{e_n:n\in\mathbb{Z}\}$ is dense in $\mathcal{X}$ (see \cite[Theorem~4.12]{grosse2011linear} and its proof). On the other hand, let $B$ be hypercyclic. Let $N\geq 0$ and $\varepsilon>0$, and choose a corresponding $\delta>0$ by \labelcref{eq:LC1}. By hypercyclicity there is some $x=(x_n)_n\in\mathcal{X}$ and some $k>2N$ such that
\begin{equation}\label{eq:hc}
\Big\|x-2\sum_{|n|\leq N}e_n\Big\|<\delta \hspace*{0.5cm} \text{and}\hspace*{0.5cm} \Big\|B^kx-2\sum_{|n|\leq N}e_n\Big\|<\delta.
\end{equation}
Since \labelcref{eq:LC1} implies the continuity of the coordinate functionals, and by choosing $\delta$ smaller if necessary, we can also have that, for $|n|\leq N$,
\begin{equation}\label{eq:hc2}
|x_n-2|<1 \hspace*{0.5cm} \text{and}\hspace*{0.5cm} |x_{n+k}-2|<1.
\end{equation}
Now let $|j|\leq N$. Then $|j-k|>N$, so that by the first inequality in \labelcref{eq:hc2},
\[
\Big|\Big[B^kx-2\sum_{|n|\leq N}e_n\Big]_{j-k}\Big| = |x_j|>1,
\]
where $[x]_n$ denotes the $n$-th element in the sequence $x$. Then \labelcref{eq:LC1} together with the second inequality in \labelcref{eq:hc} implies that $\|e_{j-k}\|<\varepsilon$. 

Similarly, since $|j+k|>N$, the second inequality in \labelcref{eq:hc2} gives that
\[
\Big|\Big[x-2\sum_{|n|\leq N}e_n\Big]_{j+k}\Big| = |x_{j+k}|>1,
\]
and hence, by \labelcref{eq:LC1} together with the first inequality in \labelcref{eq:hc}, $\|e_{j+k}\|<\varepsilon$. 

This proves the necessity in the case of hypercyclicity. The proof in the mixing case is done analogously.
\end{proof}

In exactly the same way we obtain the unilateral case. Here, when applying \cref{th: dissipative,t-dissipative-mix}, one considers the measure $\mu$ on $\mathcal{P}(\mathbb{N})$ given by $\mu(\{n\})=\|e_n\|$, $n\in \mathbb{N}$, the transformation $f(n)=n+1$, and $W=\{1\}$, where one notes that $f^{-n}(W)=\varnothing$ for all $n\geq 1$. A further simplification can be achieved by \cite[Lemma~4.2]{grosse2011linear}, which shows that it is enough to consider $j=0$. Alternatively, one may give a direct proof exactly as for the bilateral case, using the proof of \cite[Theorems~4.3 and 4.5]{grosse2011linear} for the sufficiency.

\begin{theorem}\label{c:unil}
Let $\mathcal{X} \subseteq \mathbb{K}^\mathbb{N}$ be a Banach space of sequences in which $\text{\emph{span}}\{e_n:n\in\mathbb{N}\}$ is dense and such that, for every $\varepsilon>0$, there is some $\delta>0$ such that
\[
\|x\|<\delta \Longrightarrow \forall n\in\mathbb{N}, \ \min(|x_n|,\|e_n\|)<\varepsilon.
\]
Suppose that the backward shift $B$ is an operator on $\mathcal{X}$. 

\emph{(a)} $B$ is hypercyclic if, and only if, there is an increasing sequence of integers $(n_k)_{k\geq 1}$ such that
 \[
e_{n_k}\to 0.
\]

\emph{(b)} $B$ is mixing if, and only if, 
 \[
e_{k}\to 0.
\]
\end{theorem}

Previously, the best result in the literature had required that the coordinate projections $x\mapsto x_n e_n$, $n\in\mathbb{Z}$, were equicontinuous on $\mathcal{X}$ (\cite[Theorems 6 and 7]{grosse2000shifts}; see also \cite[Theorems 4.3 and ~4.12]{grosse2011linear} when $(e_n)_n$ is even a basis in $\mathcal{X}$). Since $\min(a,b)\leq \sqrt{ab}$ for $a,b\geq 0$, condition \labelcref{eq:LC1} is less restrictive. The following example shows that we have a strict improvement.

\begin{example}
Let $\mathcal{X}=\{x=(x_n)_{n\geq 1} : |\frac{x_{n+1}}{n+1}- \frac{x_{n}}{n}|\to 0\}$ with norm $\|x\|=|x_1|+ \sup_{n\geq 1}|\frac{x_{n+1}}{n+1}- \frac{x_{n}}{n}|$. Then $\mathcal{X}$ satisfies the hypotheses of \cref{c:unil}, where the density of $\text{span}\{e_n:n\in\mathbb{N}\}$ follows from \cite[Example 5.2.14]{wilansky1984}, and it is not difficult to see that $B$ is an operator on $\mathcal{X}$. On the other hand, the coordinate projections are not equicontinuous because $x=(n\sum_{k=1}^n\frac{1}{k})_n\in \mathcal{X}$ but $\|x_ne_n\| = \frac{|x_n|}{n}\to\infty$. Thus, $B$ is a hypercyclic (even mixing) operator that is outside the scope of \cite{grosse2000shifts}.
\end{example}

A similar example can be given in the bilateral case.

\section{Dissipative systems with bounded distortion}\label{sec: bounded distortion}

Motivated by the recent paper \cite{daniello2022shiftlike},
we show in this section that for invertible dissipative systems with bounded distortion, as in the case of weighted shifts, an operator being mixing is equivalent to it satisfying Kitai's Criterion.

Let $(X, \mathcal{B},\mu, f)$ be a dissipative system generated by $W$, and suppose that $f:X\to X$ is bijective and bimeasurable. We assume that the system $(X, \mathcal{B},\mu, f)$ is of \textit{bounded distortion}, that is, $0<\mu(W)<\infty$ and there exists $K>0$ such that
\[
\frac{1}{K}\frac{\mu(f^k(W))}{\mu(W)}\leq \frac{\mu(f^k(B))}{\mu(B)} \leq K \frac{\mu(f^k(W))}{\mu(W)}
\]
for all $k\in \mathbb{Z}$ and $B\in \mathcal{B}(W)=\{A\cap W:A\in \mathcal{B}\}$.

Following \cite{daniello2022shiftlike} we first consider the space $\mathcal{X}=L^p(X)$, $1\leq p<\infty$. Then let $B_w$ be the weighted shift operator on $\ell^p(\mathbb{Z})$ with weights given by
\[
w_k=\Big(\frac{\mu(f^{k-1}(W))}{\mu(f^k(W))}\Big)^{\frac{1}{p}}, k\in\mathbb{Z};
\]
where $B_w$ is defined by $B_w(x_k)_{k\in \mathbb{Z}}= (w_{k+1}x_{k+1})_{k\in \mathbb{Z}}$.

 The relationship between $T_f$ and $B_w$ is studied in \cite[Theorem~M]{daniello2022shiftlike}. In particular, it is shown there that $B_w$ is mixing if, and only if, $T_f$ is mixing. The next proposition shows that we cannot drop the assumption that the system has bounded distortion in order to conclude that $T_f$ is mixing implies that $B_w$ is mixing.

\begin{proposition}\label{contraexemplo mixing theorem M}
    There exists an invertible dissipative system $(X,\mathcal{B},\mu,f)$ such that the composition operator $T_f$ on $L^p(X)$, $1\leq p<\infty$, is mixing, while the weighted shift $B_w$ on $\ell^p(\mathbb{Z})$, with weights given by \begin{equation}\label{pesos}
        w_k=\Big(\frac{\mu(f^{k-1}(W))}{\mu(f^k(W))}\Big)^\frac{1}{p}, k\in\mathbb{Z},
    \end{equation}
    is well defined but not mixing.
\end{proposition}

\begin{proof}
    Let $(X,\mathcal{B},\mu,f)$ be the invertible measurable system from the proof of \cref{principal}. It is dissipative with generator $W=\{0\}\times [0,1]$. Let $\mathcal{X}=L^p(X)$, $1\leq p<\infty$. By construction, the weights $w_k$, $k\in\mathbb{Z}$, are bounded, so that $B_w$ is a weighted shift on $\ell^p(\mathbb{Z})$. 

    Let $D\subseteq \mathbb{N}$ be the (infinite) set of numbers $k\geq 1$ such that 
    \[
    \mu(f^k(W))\geq 1,
    \]
    that is, $D=\{3,7,14,22,\ldots\}$.
    
    We claim that 
    \[
    \prod_{j=1}^{k} w_j\leq 1
    \]
    for every $k\in D$. Indeed, by \eqref{pesos} we have that
    \[
    \prod_{j=1}^{k} w_j=\Big(\frac{\mu(W)}{\mu(f(W))}\Big)^\frac{1}{p}\cdot \Big(\frac{\mu(f(W))}{\mu(f^2(W))}\Big)^\frac{1}{p}\cdots \Big(\frac{\mu(f^{k-1}W)}{\mu(f^k(W))}\Big)^\frac{1}{p}=\Big(\frac{\mu(W)}{\mu(f^k(W))}\Big)^\frac{1}{p}\leq 1,
    \]
    since $\mu(W)=\lambda([0,1])=1$ and $k\in D$.

    The previous claim implies that 
    \[
    \prod_{j=1}^{k} w_j\leq 1
    \]
    for infinitely many $k\geq 1$. By \cite[Theorem~4.13]{grosse2011linear}, this shows that $B_w$ is not mixing.
\end{proof}

We now show that, for very general spaces $\mathcal{X}$, and for invertible dissipative systems with bounded distortion, mixing implies Kitai's Criterion.

\begin{theorem}\label{theorem kitai bounded dist}
    Let $(X,\mathcal{B},\mu,f)$  be an invertible dissipative system of bounded distortion. Suppose that $\mathcal{X}$ satisfies \emph{(LH1)}, \emph{(LH2)}, \emph{(LC1)} and \emph{(LC2)}. The following are equivalent:
    \begin{enumerate}
        \item[\emph{(a)}] $T_f$ is mixing;
        \item[\emph{(b)}] $T_f$ satisfies Kitai's Criterion;
        \item[\emph{(c)}] $f$ satisfies  
        \[
    \mu \big(f^n(W)\big) \to 0 \hspace*{0.5cm} \text{ and } \hspace*{0.5cm} \mu \big(f^{-n}(W)\big) \to 0.
    \]
    \end{enumerate}
\end{theorem}

\begin{proof}
Suppose that (c) holds. Let $m\in\mathbb{Z}$ and $A\subseteq W_m$ be a measurable set with finite measure. Then the set $B:=A$ satisfies by the hypothesis that
     \[
    \mu(f^n(B))\leq\mu (f^n(W_m))= \mu(f^{m+n}(W))\to  0
    \]
    and 
    \[
    \mu(f^{-n}(B))\leq\mu ( f^{-n}(W_m))= \mu(f^{m-n}(W))\to  0.
    \]
    By \cref{thm: kitai characterization diss}, we can conclude that $T_f$ satisfies Kitai's Criterion.
    
    This in turn implies that $T_f$ is mixing.

    Now suppose that $T_f$ is mixing. By \cref{t-dissipative-mix}, condition (MSC) holds. Since, by assumption, $0<\mu(W)<\infty$, this condition implies that there is a sequence of measurable sets $B_n\subseteq W$ such that $\mu(B_n)\to \mu(W)$,
\[
				 \mu(f^{-n}(B_n))\to 0 \hspace*{0.5cm}\text{and}\hspace*{0.5cm} \mu(f^{n}(B_n))\to 0.
				\] 
But then, since the system has bounded distortion, there is some $K>0$ such that, for all $k\in \mathbb{Z}$,
\[
\frac{\mu(f^k(W))}{\mu(W)}\leq K \frac{\mu(f^k(B_n))}{\mu(B_n)}.
\]
Altogether we obtain that
\[
 \mu(f^{-n}(W))\to 0 \hspace*{0.5cm}\text{and}\hspace*{0.5cm} \mu(f^{n}(W))\to 0,
\]
and hence (c).
\end{proof}

As a consequence we obtain the version of \cite[Theorem M]{daniello2022shiftlike} for Kitai's Criterion. This follows from the fact that the result holds for the property of mixing, and 
because mixing and satisfying Kitai's Criterion are equivalent conditions for weighted shifts.

\begin{corollary}\label{corollary kitai bounded}
Let $(X,\mathcal{B},\mu,f)$  be an invertible dissipative system of bounded distortion. Let $1\leq p<\infty$, and let $B_w$ be the weighted shift on $\ell^p(\mathbb{Z})$ with weights given by \eqref{pesos}. Then $T_f$ satisfies Kitai's Criterion on $\mathcal{X}=L^p(X)$ if and only if $B_w$ satisfies Kitai's Criterion on $\ell^p(\mathbb{Z})$.
\end{corollary}

We note that the corollary can also be obtained directly, using the methods of \cite[Section 5]{daniello2022shiftlike}.

\begin{acknowledgement}
    The first author was supported by the São Paulo Research Foundation (FAPESP), grants 2021/02672-2 and 2023/03661-0; the second author was supported by the Fonds de la Recherche Scientifique - FNRS under grant n\textsuperscript{o} CDR J.0078.21.
\end{acknowledgement}

\end{document}